\documentclass{amsart}
\usepackage[utf8]{inputenc}
\usepackage{epsfig}
\usepackage{amssymb}
\usepackage{amscd}
\usepackage[matrix,arrow]{xy}
\usepackage{graphicx}
\usepackage{amsmath}
\usepackage{amsthm}
\usepackage{amsfonts, bbm, mathrsfs}
\usepackage{color, enumerate}
\usepackage{hyperref}
\hypersetup{colorlinks=true,citecolor=NavyBlue,linkcolor=Brown,urlcolor=Orange}
\usepackage[usenames,dvipsnames,svgnames,table]{xcolor, pstricks}
\usepackage{tikz, multicol}
\usepackage{tikz-cd}
\usepackage{float}
\usepackage{mathtools}
\usepackage{comment}
\usepackage{enumitem}

\usepackage{geometry}
\newtheorem*{thm*}{Theorem}
\newtheorem*{prop*}{Proposition}
\newtheorem*{cor*}{Corollary}
\newtheorem{thm}{Theorem}[section]
\newtheorem{prop}[thm]{Proposition}
\newtheorem{cor}[thm]{Corollary}
\newtheorem{lemma}[thm]{Lemma}
\numberwithin{equation}{section}

\theoremstyle{definition}

\newtheorem{defn}[thm]{Definition}
\newtheorem{eg}[thm]{Example}

\newtheorem{rmk}[thm]{Remark}

\def\bR{\mathbb{R}}
\def\bC{\mathbb{C}}
\def\bZ{\mathbb{Z}}

\def\bP{\mathbb{P}}

\def\bL{\mathbb{L}}

\def\cD{\mathcal{D}}
\def\cO{\mathcal{O}}

\def\cN{\mathcal{N}}

\def\cP{\mathcal{P}}

\def\Stab{\mathrm{Stab}}
\def\Aut{\mathrm{Aut}}

\def\Hom{\mathrm{Hom}}

\def\Coh{\mathrm{Coh}}

\def\cat{\mathrm{cat}}
\def\pol{\mathrm{pol}}
\def\top{\mathrm{top}}

\def\tr{\mathrm{tr}}
\def\Pic{\mathrm{Pic}}

\def\k{\mathbf{k}}
\DeclareMathOperator{\RHom}{RHom}
\def\ph{h^\pol} 

\def\ra{\rightarrow}
\def\bs{\backslash}
\def\ep{\epsilon}

\title{Categorical polynomial entropy}
\author{Yu-Wei Fan \and Lie Fu \and Genki Ouchi}

\begin{document}

\maketitle

\begin{abstract}
    For classical dynamical systems, the polynomial entropy serves as a refined invariant of the topological entropy. In the setting of categorical dynamical systems, that is, triangulated categories endowed with an endofunctor, we develop the theory of categorical polynomial entropy, refining the categorical entropy defined by Dimitrov, Haiden, Katzarkov, and Kontsevich. We justify this notion by showing that for an automorphism of a smooth projective variety, the categorical polynomial entropy of the pullback functor on the derived category coincides with the polynomial growth rate of the induced action on cohomology. We also establish in general a Yomdin-type lower bound for the categorical polynomial entropy of an endofunctor in terms of the induced endomorphism on the numerical Grothendieck group of the category. As examples, we compute the categorical polynomial entropy for some standard functors like shifts, Serre functors, tensoring line bundles, automorphisms, spherical twists, P-twists, and so on, illustrating clearly how categorical polynomial entropy refines the study of categorical entropy and enables us to study the phenomenon of categorical trichotomy. A parallel theory of polynomial mass growth rate is developed in the presence of Bridgeland stability conditions.
\end{abstract}

\setcounter{tocdepth}{1}
\tableofcontents

\section{Introduction}
\subsection{Background}
A topological dynamical system consists of a compact Hausdorff topological space $X$ and a continuous self-map $f\colon X\ra X$.
The \emph{topological entropy} of $f$, denoted by $h_\top(f)$, is a non-negative real number (possibly infinite) measuring the complexity of the system, by looking at the asymptotic behaviour of the iterations of the map $f$.
More precisely, using an auxiliary metric\footnote{See \cite{MR175106} for a metric-free (and equivalent) definition.} on $X$, the topological entropy $h_\top(f)$, which is independent of the metric, is defined to be the exponential growth rate of certain positive number $\mathrm{cov}(f, n, \ep)$ measuring the $\epsilon$-separation property of the $n$-th iteration of $f$:
\[
h_\top(f) \coloneqq \lim_{\ep\ra0}\limsup_{n\ra\infty}
\frac{\log\mathrm{cov}(f,n,\ep)}{n}.
\]
We refer to \cite{KatokHasselblatt, GromovAsterisque, CanPR} for more details. When restricting to the category of compact K\"ahler manifolds, the topological entropy can be understood via the cohomology; this is the content of the following fundamental result due to Gromov (the upper bound) and Yomdin (the lower bound). 

\begin{thm}[\cite{Gromov,Yomdin}]\label{thm:GromovYomdin}
Let $f$ be a $\mathcal{C}^\infty$ self-map of a compact differentiable manifold $X$.
Then
\[
h_\top(f)\geq\log\rho(f^*),
\]
where  $\rho(f^*)$ denotes the spectral radius of the induced endomorphism $f^*$ on the cohomology $H^*(X,\bC)$.
Moreover, if $X$ is a compact K\"ahler manifold and $f\colon X\ra X$ is a surjective holomorphic map, then
\[
h_\top(f)=\log\rho(f^*).
\]
\end{thm}
In the compact K\"ahler setting, a more refined study involves the so-called \textit{dynamical degrees} (see the work of Dinh--Sibony \cite{DinhSibonyJAMS, DinhSibonyAnnals}); some important facts on this matter are collected in Section \ref{sec:DynamicalDegree}. We refer to the survey of Oguiso \cite{OguisoICM} and the references therein for an overview of interesting examples and results on dynamical systems in algebraic geometry. 

\medskip

Recently, there has been a growing interest in dynamical systems with vanishing topological entropy, highlighting the study of the so-called \emph{polynomial entropy}; see for instance \cite{BernardLabrousse,CanPR,HauseuxLeRoux,Labrousse12,Labrousse12balls,LabrousseMarco,Marco13,Marco18}.
More precisely, given a topological dynamical system $(X, f)$ with $h_\top(f)=0$, the polynomial entropy is defined to be the polynomial growth rate of the number $\mathrm{cov}(f,n,\ep)$:
\[
h_\pol(f) \coloneqq \lim_{\ep\ra0}\limsup_{n\ra\infty}
\frac{\log\mathrm{cov}(f,n,\ep)}{\log (n)}.
\]
This provides a refined/secondary invariant for the dynamical system.
In the spirit of Theorem \ref{thm:GromovYomdin}, given an automorphism of a compact K\"ahler manifold with null entropy, Cantat--Paris-Romaskevich \cite[Theorem 2.1 and Theorem 4.1]{CanPR} established some upper bounds for the polynomial entropy in terms of the induced action on cohomology, together with some topological data of the manifold. If certain dynamical degree is 1, the study of a secondary invariant called \textit{polynomial dynamical degrees} is initiated by Lo Bianco \cite[Section 1.3]{LoB}. In Section \ref{sec:PolDynamicalDegree} we extend his definition and results to the general case where the dynamical degree is arbitrary. We refer to \cite{CanPR} and references therein for other results on polynomial entropy in the compact K\"ahler setting.

\medskip

The \emph{categorical} counterpart of the notion of topological entropy was introduced by Dimitrov, Haiden, Katzarkov, and Kontsevich \cite{DHKK}.
The definition of \emph{categorical entropy} is motivated by the profound connection between the Teichm\"uller theory and the theory of stability conditions on triangulated categories, established recently in \cite{GMN,BS15,HKK,DHKK}.
Let $\cD$ be a triangulated category and let $F\colon\cD\ra\cD$ be a triangulated functor that we always assume not to be virtually zero, that is, $F^n\not\cong 0$ for all $n>0$. We view the pair $(\cD, F)$ as a \textit{categorical} dynamical system.
The categorical entropy of $F$, denoted by $h_t(F)$, which is a real function in one real variable, is defined to be the exponential growth rate of certain positive number $\delta_t(G, F^n(G))$ measuring the complexity of the image of a split generator $G$ of $\cD$ under the $n$-th iteration of $F$, with respect to $G$ itself:
\[
h_t(F) \coloneqq \lim_{n\ra\infty} \frac{\log \delta_t(G,F^n(G))}{n}\in[-\infty,\infty).
\]
Note that the categorical entropy is independent of the choice of the split generator.
We review the definition and basic properties of $\delta_t$ and $h_t$ in Section \ref{section:basicnotions}.
The value of the categorical entropy at $t=0$ is often of particular interest, and is denoted by
\[
h^\cat(F) \coloneqq h_0(F)\in \mathbb{R}_{\geq 0}.
\]
The connection to the classical dynamical system justifies the notion of categorical entropy:
it is proved in \cite[Theorem 2.12]{DHKK} under some technical condition and by Kikuta and Takahashi \cite{KiTa} in full generality that if $f\colon X\ra X$ is a surjective endomorphism of a smooth projective complex variety and
$\bL f^*\colon\cD^b(X)\ra\cD^b(X)$ is the derived pullback functor on the bounded derived category of coherent sheaves on $X$, then
\[
h^\cat(\bL f^*) = h_\top(f).
\]
See Theorem \ref{thm:KTO} for a more general version due to Ouchi \cite{Ouchi}.
\subsection{New invariants}
The goal of the present article is to lay the foundation of the theory of \emph{categorical polynomial entropy}, which should serve as the categorical counterpart of the aforementioned notion of polynomial entropy. Inspired by \cite{CanPR},
we propose to define it as the polynomial growth rate of $\delta_t(G,F^n(G))$:
\[
h_t^\pol(F)\coloneqq\limsup_{n\ra\infty}\frac{\log\delta_t(G,F^n(G))-nh_t(F)}{\log (n)}\in[-\infty,\infty].
\]
Note that $h_t^\pol(F)$ is independent of the choice of the split generator $G$.
Moreover, the definition makes sense even if the categorical entropy $h_t(F)$ does not vanish, as long as $h_t(F)\neq-\infty$.
We refer to Section \ref{section:catpolyentropy} for some basic properties of $h_t^\pol(F)$. Its value at $t=0$ is denoted by
\[
\ph(F)\coloneqq \ph_0(F)
\]
and is well-defined since $h_0(F)\geq0$ is a real number.

If the triangulated category $\cD$ admits a Bridgeland stability condition $\sigma$, the notion of \textit{mass growth} of an endofunctor $F$, denoted by $h_{\sigma, t}(F)$,  was defined in \cite{DHKK} and studied in \cite{Ikeda}, as a comparable invariant of categorical entropy. Mass growth is the categorical analogue of the \textit{volume growth} for classical dynamical systems equipped with a Riemannian metric, see Yomdin \cite{Yomdin}. We develop in Section \ref{sec:MassGrowth} the basic theory of \textit{polynomial mass growth} for a categorical dynamical system endowed with a stability condition $(\cD, F, \sigma)$, in a parallel way to the theory of categorical polynomial entropy. The polynomial mass growth rate is invariant under deformation of the stability condition inside the stability manifold (Lemma \ref{lemma:DeformStab}).

Let us put those invariants into perspective, with the new ones colored in blue:
\begin{center}
\begin{tabular}{c|c|c}
    & $\substack{\text{Topological system }\\ (X, f\colon X\to X)}$ & $\substack{\text{Categorical system }\\ (\cD, F\colon \cD\to \cD)}$ \\
    \hline
    Entropy & $h_\top(f)$ & $h_t(F)$\\
   Polynomial entropy & $h_\pol(f)$ & \color{blue}{$\ph_t(F)$}\\
   Volume/mass growth & $\operatorname{lov}(f)$& $h_{\sigma, t}(F)$\\
   Polynomial volume/mass growth & $\operatorname{povol}(f)$ & \color{blue}{$\ph_{\sigma, t}(F)$}
\end{tabular}
\end{center}

\subsection{Gromov--Yomdin-type results}
As a justification of our definition, when the categorical dynamical system comes from a classical one, namely, $$(\cD=\cD^b(X), F=\mathbb{L}f^*)$$ for a surjective endomorphism $f$ of a smooth projective variety $X$ defined over an algebraically closed field,  it turns out that the categorical polynomial entropy behaves better than the (topological) polynomial entropy, in view of Theorem~\ref{thm:GromovYomdin}.
More precisely, in Theorem \ref{thm:EntropyOfMorphism}, we show that the categorical polynomial entropy of $\mathbb{L}f^*$  can be computed using polynomial dynamical degrees (Definition \ref{def:PolDynDeg}), and is equal to the polynomial growth rate, in the generalized sense of Definition \ref{def:polygrowthratelinearmap}, of the induced action on the numerical Grothendieck group (or the cohomology in the complex setting) of $X$. Here in the introduction, let us only state the following special case over the complex numbers and with the additional assumption that the entropy vanishes.
 
\begin{thm}[see Corollaries \ref{cor:polycatautomorphism} and \ref{cor:MassGrowthMorphism}]
Let $f$ be a surjective endomorphism of a smooth projective complex variety $X$. Assume that the topological entropy of $f$ is zero. Then $\ph_t(\mathbb{L}f^*)$ is a constant function with value
$$\ph(\mathbb{L}f^*)=\lim_{n\to \infty} \frac{\log\|(f^n)^*\|}{\log (n)}=s(f^*),$$
where $f^*\in \operatorname{End}(H^*(X, \mathbb{Z}))$ is the induced endomorphism on cohomology and $s(f^*)+1$ is the maximal size of its Jordan blocks. If moreover $\cD^b(X)$ admits a numerical stability condition $\sigma$, then $\ph_{\sigma, 0}(\mathbb{L}f^*)$ is also equal to the above quantities.
\end{thm}

One should compare the above theorem to the estimates on the (topological) polynomial entropy of automorphisms of compact K\"ahler manifold established by Cantat and Paris-Romaskevich \cite[Theorem 2.1 and Theorem 4.1]{CanPR}.

In general, we establish the following Yomdin-type lower bound for the categorical polynomial entropy.

\begin{thm}[see Propositions~\ref{prop:GYLowerBoundPolEntropy} and \ref{prop:GYLowerBoundMassGrowth}]
Let $F\colon\cD\ra\cD$ be an endofunctor on a triangulated category $\cD$, and let $\cN(F)$ be the induced endomorphism of the numerical Grothendieck group $\cN(\cD)$. 
    \begin{enumerate}[label=(\roman*)]
        \item If $\cD$ is saturated, then $h^\cat(F)\geq\log \rho(\cN(F))$ by \cite[Theorem~2.13]{KST18}. If the equality holds (for instance when $h^\cat(F)=0$), then
            \[
            h^\pol(F) \geq s(\cN(F)),
            \]
        where $s(\cN(F))$ is the polynomial growth rate of the linear map $\cN(F)$ (cf.~Definition~\ref{def:polygrowthratelinearmap}).
        \item If $\cD$ admits a numerical stability condition $\sigma$, then $h_{\sigma, 0}(F)\geq\log \rho(\cN(F))$ by \cite[Theorem~1.2]{Ikeda}. If the equality holds (for instance when $h_{\sigma, 0}(F)=0$), then
            \[
            \ph_{\sigma, 0}(F) \geq s(\cN(F)).
            \]
    \end{enumerate}
\end{thm}
In Proposition~\ref{prop:hereditaryalgebraGY} and Corollary \ref{cor:hereditaryalgebraGYMass}, we show that the lower bounds are attained for any autoequivalence of the bounded derived category of a hereditary finite dimensional $\bC$-algebra.

\subsection{Categorical trichotomy}
The trichotomy for birational automorphisms of projective surfaces is one of the most fascinating phenomena in classical (algebraic) dynamical systems, see Cantat \cite{Cantat99} and Diller--Favre \cite{DillerFavre}. Recall that given a smooth projective surface $X$, a birational self-map $f\colon X\dashrightarrow X$ falls into three possibilities: \textit{elliptic}, \textit{parabolic}, and \textit{loxodromic}. More precisely, let $H$ be a polarization on $X$, which allows one to define the algebraic degree $\deg_H(f)$, then $f$ is called
\begin{itemize}
    \item elliptic, if $\deg_H(f^n)$ is bounded in $n$; in this case, $f$ conjugates to a map that is virtually isotopic to the identity;
    \item parabolic, if $\deg_H(f^n)$ has polynomial growth in $n$; in this case, $f$ preserves a  fibration;
    \item loxodromic, if $\deg_H(f^n)$ has exponential growth in $n$; in this case, $f$ has positive entropy and there are strong restrictions to the surface.
\end{itemize}
We refer to \cite{DillerFavre} for more details. 

By combining the force of categorical entropy and categorical polynomial entropy, we are naturally led to study the analogous \textit{categorical} trichotomy by calling a categorical dynamical system $(\cD, F)$
\begin{itemize}
    \item elliptic, if $h^\cat(F)=\ph(F)=0$;
    \item parabolic, if $h^\cat(F)=0$, $\ph(F)>0$;
    \item loxodromic, if $h^\cat(F)>0$.
\end{itemize}
We are indeed able to pinpoint some natural examples illustrating this phenomenon. The first instance we discover concerns the autoequivalences of derived categories of elliptic curves, refining Kikuta's result \cite[Section 3.2]{Kikuta}.
As before, $\mathcal{N}(F)$ is the induced endomorphism on the numerical Grothendieck group, which is 2-dimensional in the curve case.
\begin{thm}[see Theorem~\ref{thm:ellcurveclassify}]
Let $\cD$ be the bounded derived category of a smooth projective curve of genus 1 defined over an algebraically closed field, 
and let $F\colon\cD\ra\cD$ be an autoequivalence.
Then
    \begin{enumerate}[label=(\roman*)]
        \item $h^\pol(F)=h^\cat(F)=0$ if and only if
        $\cN(F)$ is elliptic (i.e.~$|\tr(\cN(F))|<2$) or 
        $\cN(F)=\pm\mathrm{id}$.
        
        \item $h^\pol(F)>0$ and $h^\cat(F)=0$ if and only if
        $\cN(F)$ is parabolic (i.e.~$|\tr(\cN(F))|=2$)
        and $\cN(F)\neq\pm\mathrm{id}$. In this case, $\ph(F)=1$.
        
        \item $h^\cat(F)>0$ if and only if 
        $\cN(F)$ is hyperbolic (i.e.~$|\tr(\cN(F))|>2$). In this case, $\ph(F)=0$.
    \end{enumerate}
\end{thm}

A second instance of such categorical trichotomy is proved for the 3-Calabi--Yau category associated to the $A_2$-quiver, by looking at the group of autoequivalences generated by the two natural spherical twists, see Section~\ref{sec:sphericaltwistquiver} for the precise statement.

\subsection{Other examples}
There are many examples of autoequivalences of triangulated categories that give natural non-trivial dynamical systems with zero categorical entropy.
Such examples include tensoring line bundles on smooth projective varieties, spherical twists along spherical objects (analogue of Dehn twists along Lagrangian spheres, via Homological Mirror Symmetry) \cite{SeidelThomas}, and $\bP$-twists along $\bP$-objects (analogue of Dehn twists along Lagrangian complex projective space) \cite{HuybrechtsThomas}.
We show in Section~\ref{sec: Examples} that the categorical polynomial entropy provides a non-trivial invariant for these autoequivalences.
For example, the categorical polynomial entropy of the autoequivalence of tensoring a line bundle in the derived category of a smooth projective variety encodes the \textit{positivity} property of the line bundle.

\begin{thm}[see Proposition \ref{prop:TensorLineBundle}, Theorem \ref{thm:TensorNefLineBundle} and Corollary \ref{cor:NumericalNature}]
Let $L$ be a line bundle on a smooth projective variety $X$ defined over an algebraically closed field. The categorical polynomial entropy of the functor $-\otimes L$ on $\cD^b(X)$ is a constant function whose value depends only on the numerical class of $L$, and satisfies
\[\nu(L)\leq \ph(-\otimes L) \leq \dim(X),\]
where $\nu(L)\coloneqq \max\{m\mid c_1(L)^m\not\equiv 0\}$ is the numerical dimension of $L$.\\
Moreover, if $L$ or $L^{-1}$ is nef, then  $$\ph(-\otimes L)=\nu(L).$$
\end{thm}

We refer to Sections \ref{sec:sphericaltwists}, \ref{sec:sphericaltwistquiver}, and \ref{sec:P-twists} for results on the categorical polynomial entropy of spherical twists and $\bP$-twists.

Finally, we observe the following Gromov--Yomdin-type equality in the curve case, which is reminiscent of \cite{Kikuta}.
\begin{thm}[See Proposition \ref{prop:StarndardAutCurve} and Theorem \ref{thm:ellcurveclassify}]
Let $C$ be a smooth projective curve defined over an algebraically closed field. Let  $F$ be any autoequivalence of the bounded derived category $\cD^b(C)$. Then the categorical polynomial entropy of $F$ is equal to the polynomial growth of the induced action on the (2-dimensional) cohomology/numerical Grothendieck group:
\[
h^\pol(F) = s (\cN(F))\in \{0, 1\}.
\]
\end{thm}


\noindent\textbf{Convention:} $\k$ is a base field. In this paper, all categories, as well as functors between them, are assumed to be $\k$-linear. Functors between triangulated categories are always assumed to be triangulated (i.e.~they preserve distinguished triangles) and not \textit{virtually} zero (i.e.~no iteration is the zero functor).  A triangulated category is called \textit{saturated} if it is equivalent to the homotopy category of a triangulated saturated $A_\infty$-category; or equivalently, it admits a dg-enhancement which is triangulated, smooth and proper. Note that the notion of  saturatedness here is stronger than that of Bondal--Kapranov \cite{BondalKapranov}.\\

\noindent\textbf{Acknowledgement:} We would like to thank Serge Cantat, Zhi Jiang, John Lesieutre, Olga Paris-Romaskevich, Steffen Sagave, and Junyi Xie for helpful discussions. Lie Fu and Genki Ouchi also want to thank the fourth edition of \textit{Japanese-European Symposium on Symplectic Varieties and Moduli Spaces} in Z\"urich 2019, where the initial ideas of this project took form. Lie Fu is supported by the Radboud Excellence Initiative program. Genki Ouchi is supported by Interdisciplinary Theoretical and Mathematical Sciences Program (iTHEMS) in RIKEN and JSPS KAKENHI Grant number 19K14520.

\section{Categorical polynomial entropy}
\label{section:catpolyentropy}
In this section, inspired by the categorical formalism from \cite{DHKK} and geometrical considerations from \cite{CanPR}, we develop the basic theory of categorical polynomial entropy of an endofunctor of a triangulated category, as a secondary invariant of a categorical dynamical system, refining the categorical entropy.
\subsection{Basic notions}
\label{section:basicnotions}
\begin{defn}[Complexity function {\cite[Definition 2.1]{DHKK}}] 
\label{def:Complexity}
 Let $\cD$ be a triangulated category. Given any two objects $M, N\in \cD$, we define the following functions in a real variable $t$, which take values in $\mathbb{R}_{\geq 0}\cup \{\infty\}$:
\begin{itemize}
    \item the \textit{complexity function} of $N$ with respect to $M$:
$$\delta_t(M,N)\coloneqq\inf\left\{\sum_{k=1}^m e^{n_kt}~\mid \substack{0=A_0\to A_1\to \dots \to A_m=N\oplus N',\\
\text{for some } N'\in \cD,\\ \operatorname{cone}(A_{k-1}\to A_k)\simeq M[n_k], \forall k\in \mathbb{Z}} \right\}.$$
By convention, $\delta_t(M, N)=0$ if $N\simeq 0$ and $\delta_t(M,N)=\infty$ if and only if $N$ is not in the thick triangulated subcategory generated by $M$.
\item The $\operatorname{Ext}$-\textit{distance function} from $M$ to $N$: $$\epsilon_t(M, N)\coloneqq\delta_t(\k, \RHom(M, N))=\sum_{k\in \mathbb{Z}} \dim_\k \Hom(M, N[k])\cdot e^{-kt}.$$
\item Their values at $t=0$ are of special interest and we denote
$$\delta(M, N)\coloneqq\delta_0(M, N);\quad \epsilon(M, N)\coloneqq\epsilon_0(M, N),$$
which take values in $[1, \infty]$ when $N\not\simeq0$.
\end{itemize} 
\end{defn}

\begin{rmk}\label{rmk:PropOfComplexity}
We collect some basic properties of the functions in Definition \ref{def:Complexity}. See \cite[Proposition 2.3, Theorem 2.7]{DHKK} for details.
\begin{enumerate}[label=(\roman*)]
    \item When $\cD$ is saturated, the complexity function and the $\operatorname{Ext}$-distance function are controlled by each other: there exist functions $C_1, C_2\colon \mathbb{R}\to \mathbb{R}_{>0}$, such that for any $M, N\in \cD$, we have
    $$C_1(t)\epsilon_t(M, N)\leq \delta_t(M, N)\leq C_2(t)\epsilon_t(M, N).$$
    \item They both satisfy the \textit{subadditivity} condition: for any distinguished triangle $N'\to N\to N''\xrightarrow{+1}$, we have
    $$\delta_t(M, N)\leq \delta_t(M, N')+\delta_t(M, N''); $$
    $$\epsilon_t(M, N)\leq \epsilon_t(M, N')+\epsilon_t(M, N'').$$
    \item $\delta_t$ satisfies the (multiplicative) \textit{triangle inequality}:
    $$\delta_t(M_1, M_3)\leq\delta_t(M_1,  M_2)\delta_t(M_2, M_3).$$
    \item $\delta_t$ retracts after applying a functor: for any triangulated functor $F\colon \cD\to \cD'$, we have
    $$\delta_t(M, N)\geq \delta_t(F(M), F(N)).$$
    \item If $\cD$ admits a Serre functor $\mathbf{S}$, then by Serre duality, $$\epsilon_t(M, N)=\epsilon_{-t}(N, \mathbf{S}(M)).$$
\end{enumerate}
\end{rmk}

Using the complexity functions, Dimitrov--Haiden--Katzarkov--Kontsevich \cite{DHKK} defined a categorical analogue of the notion of topological entropy, in order to measure the complexity of an endofunctor, viewed as a \textit{categorical} dynamical system. Let us recall the definition.

\begin{defn}[Categorical entropy]
\label{def:CateEntropy}
Let $\cD$ be a triangulated category with a given split generator\footnote{An object of a triangulated category is called a \textit{split generator} if the smallest thick triangulated subcategory containing it is the whole category.} $G$. Let $F\colon\cD\ra\cD$ be a (triangulated, not virtually zero) endofunctor. The \emph{categorical entropy} of $F$ is defined to be the function
$h_t(F)\colon\bR\ra[-\infty,\infty)$ in the real variable $t$, given by
\[
h_t(F)\coloneqq \lim_{n\ra\infty} \frac{\log\delta_t(G,F^n(G))}{n}.
\]
The existence of the limit is proved in \cite[Lemma 2.6]{DHKK}. As the notation suggests, $h_t(F)$ is independent of the choice of the generator \cite[Lemma 2.6]{DHKK}.
We denote
\[
h^\cat(F)\coloneqq h_0(F) \in [0, \infty).
\]
\end{defn}

Inspired by \cite{CanPR}, even in the case of ``slow" categorical dynamical systems, namely, when the categorical entropy vanishes, it is interesting to understand its complexity. To this end, we propose the following notion as a secondary invariant, which makes sense even when the categorical entropy does not vanish.
\begin{defn}[Categorical polynomial entropy]
\label{def:CatPolEntropy}
In the same setting as in Definition \ref{def:CateEntropy}, we define
the \emph{polynomial entropy} of $F$ to be the function $\ph_t(F)$ in the real variable $t$ given by
\[
\ph_t(F)\coloneqq \limsup_{n\ra\infty} \frac{\log\delta_t(G,F^n(G))-nh_t(F)}{\log (n)},
\]
where $h_t(F)$ is the categorical entropy of $F$.
It is well-defined for any $t\in\bR$ such that $h_t(F)\neq-\infty$.
In particular, it is well-defined at $t=0$ since $h_0(F)\geq0$.
We denote $\ph(F):=\ph_0(F)$.
\end{defn}

\begin{rmk}[Lower polynomial entropy]
It also makes sense to use $\liminf$ instead of $\limsup$ in Definition \ref{def:CatPolEntropy}. The function thus obtained could be called the \textit{lower polynomial entropy} of $F$, denoted by $\underline{h}^\pol_t(F)$. We do not know whether $\ph_t(F)$ and $\underline{h}^\pol_t(F)$ coincide in general. In this paper,  we will mostly focus on the study of $\ph_t$.
\end{rmk}

\subsection{Basic properties}
We show some basic properties of the categorical polynomial entropy function, in a parallel way to the analogous properties of the categorical entropy, as developed in \cite{DHKK}, \cite{Kikuta}, \cite{KiTa}. 
\begin{lemma}
\label{lem:TwoGenerators}
The definition of the categorical polynomial entropy is independent of the choice of the split generator. Moreover, for any two split generators $G, G'$ of $\cD$, we have
$$\ph_t(F) = \limsup_{n\ra\infty} \frac{\log\delta_t(G,F^n(G'))-nh_t(F)}{\log (n)}.
$$
\end{lemma}
\begin{proof}
The proof is similar to \cite[Lemma~2.6]{DHKK} by using the triangle inequality and the retraction property of $\delta_t$, recalled in Remark \ref{rmk:PropOfComplexity} (iii), (iv).
\end{proof}

With the saturatedness condition on the category, one can use the $\operatorname{Ext}$-distance function to compute the categorical polynomial entropy.
\begin{lemma}\label{lemm:EntropyInExt-distance}
Assume that $\cD$ is saturated. For any split generators $G, G'$ of $\cD$, we have that
\[
\ph_t(F) = \limsup_{n\ra\infty} \frac{\log\epsilon_t(G,F^n(G'))-nh_t(F)}{\log (n)}
\]
\end{lemma}
\begin{proof}
As in \cite[Theorem~2.7]{DHKK}, one uses Remark \ref{rmk:PropOfComplexity} (i) and Lemma \ref{lem:TwoGenerators}.
\end{proof}



\begin{lemma}\label{lemm:ConstantFunction}
Notation is as before. Assume that $\cD$ is saturated and there is a split generator $G$ of $\cD$ and an integer $M\geq 0$ such that for any $|k|\geq M$ and any $n\geq 0$, we have
\[\operatorname{Ext}^k(G, F^n(G))=0,\]
(for example, when $F$ preserves a bounded t-structure of finite cohomological dimension),
then $\ph_t(F)$ is  a constant function in $t$.
\end{lemma}

\begin{proof}
Similarly as in \cite[Lemma 2.11]{DHKK}, the vanishing hypothesis implies that 
$$|\log \epsilon_t(G,F^n(G))-\log \epsilon_0(G,F^n(G))|\leq M|t|.$$ As the categorical entropy function $h_t(F)$ is constant in $t$ (\cite[Lemma 2.11]{DHKK}), we obtain that
$$|(\log \epsilon_t(G,F^n(G))-n h_t(F))- (\log \epsilon_0(G,F^n(G))-n h_0(F))|\leq M|t|.$$
We can conclude by dividing by $\log(n)$ and taking limit.
\end{proof}

\begin{lemma}[Commutation and conjugation]\label{lemma:CommutationConjugation}
Let $\cD$ be a triangulated category and $F_1$, $F_2$ two endo-functors of $\cD$. Then 
$\ph_t(F_1F_2)=\ph_t(F_2F_1)$.\\
In particular,  if  $F_1$ is an autoequivalence, then $\ph_t(F_1F_2F_1^{-1})=\ph_{t}(F_2)$.
\end{lemma}
\begin{proof}
As shown in \cite[Lemma 2.8]{Kikuta}, $h_t(F_1F_2)=h_t(F_2F_1)$ and for any split generator $G$, 
\[\delta_t(G, (F_1F_2)^n(G))\leq \delta_t(G, F_1(G))\delta_t(G, F_2(G))\delta_t(G, (F_2F_1)^{n-1}(G)).\]
Hence 
\begin{align*}
\ph_t(F_1F_2)
&=\limsup_{n\to \infty}\frac{\log \delta_t(G, (F_1F_2)^n(G))-nh_t(F_1F_2)}{\log(n)}\\
&\leq\limsup_{n\to \infty}\frac{\log \delta_t(G, F_1(G))+\log\delta_t(G, F_2(G))+\log\delta_t(G, (F_2F_1)^{n-1}(G))-nh_t(F_2F_1)}{\log(n-1)}\\
&= \ph_t(F_2F_1).
\end{align*}
We get an equality by symmetry. The invariance by conjugation follows:
\begin{equation*}
    \ph_t(F_1F_2F_1^{-1})=\ph_t(F_2F_1^{-1}F_1)=\ph_t(F_2).\qedhere
\end{equation*}
\end{proof}

\begin{lemma}[Powers]\label{lemma:Power}
Notation is as before. For any positive integer $m$, we have $$\ph_t(F^m)\leq \ph_t(F).$$ Similarly, $\underline{h}^\pol_t(F^m)\geq \underline{h}^\pol_t(F).$
In particular, if the $\limsup$ in the definition of $\ph_t(F)$ is an actual limit, then we have equalities. 
\end{lemma}
\begin{proof}
By definition,
\[\ph_t(F^m)=\limsup_{n\to \infty}\frac{\log\delta_t(G,F^{mn}(G))-nh_t(F^m)}{\log (n)}.\]
Using the fact that $h_t(F^m)=mh_t(F)$, the right-hand side is nothing else but
\[\limsup_{n\to \infty}\frac{\log\delta_t(G,F^{mn}(G))-mnh_t(F)}{\log (nm)},\] which is less than or equal to $\ph_t(F)$. The proof for the lower polynomial entropy is similar. 
\end{proof}

\begin{lemma}[Inverse]\label{lemma:Inverse}
Let $\cD$ be a saturated triangulated category admitting a Serre functor. Then for any autoequivalence $F$ of $\cD$, we have
\[h_t(F^{-1})=h_{-t}(F); \quad \ph_t(F^{-1})=\ph_{-t}(F).\]
In particular, $h^\cat(F^{-1})=h^\cat(F)$ and $\ph(F^{-1})=\ph(F)$.
\end{lemma}
\begin{proof}
Let $G$ be a split generator and $\mathbf{S}$ the Serre functor. Then for any $n>0$, $$\epsilon_t(G, F^{-n}(G))=\epsilon_t(F^n(G), G)=\epsilon_{-t}(G, F^n(\mathbf{S}(G))),$$
where the second equality follows from Remark \ref{rmk:PropOfComplexity} (v) and the fact that all autoequivalences commute with the Serre functor. As $\mathbf{S}(G)$ is also a split generator of $\cD$, we obtain that 
\[h_t(F^{-1})=\lim_{n\to \infty} \frac{\log \epsilon_t(G, F^{-n}(G))}{n}=\lim_{n\to \infty} \frac{\log \epsilon_{-t}(G, F^n(\mathbf{S}(G)))}{n}=h_{-t}(F).\]
Consequently, combined with Lemma \ref{lemm:EntropyInExt-distance},
\begin{align*}
    \ph_t(F^{-1})&=\limsup_{n\to \infty} \frac{\log \epsilon_t(G, F^{-n}(G))-nh_t(F^{-1})}{\log(n)}\\&=\limsup_{n\to \infty} \frac{\log \epsilon_{-t}(G, F^n(\mathbf{S}(G)))-nh_{-t}(F)}{\log(n)}=\ph_{-t}(F). \qedhere
\end{align*}
\end{proof}

The following observation will be used in Section \ref{sec:LineBundle} when we study the categorical polynomial entropy of tensoring a line bundle.
\begin{lemma}[Composition of commuting functors]\label{lemma:composition}
Let $\cD$ be a saturated triangulated category admitting a Serre functor. Let $F_1$ be an autoequivalence and $F_2$ an endofunctor. Assume that $F_1F_2=F_2F_1$.
\begin{enumerate}[label=(\roman*)]
    \item If $h_t(F_1)$ is an odd function in $t$, then 
    \[h_t(F_1F_2)=h_t(F_1)+h_t(F_2).\]
    In this case, \[\ph_t(F_1F_2)\leq \ph_t(F_1)+\ph_t(F_2).\]
    The equality holds if $\ph_t(F_1)$ is an odd function.
    \item If $h^\cat(F_1)=0$, then 
    \[h^\cat(F_1F_2)=h^\cat(F_2).\]
    In this case, \[\ph(F_1F_2)\leq \ph(F_1)+\ph(F_2).\]
    If moreover $\ph(F_1)=0$, then $\ph(F_1F_2)=\ph(F_2)$.
\end{enumerate}
\end{lemma}
\begin{proof}
We only show $(i)$, as the proof of $(ii)$ is similar. By Remark \ref{rmk:PropOfComplexity} (iii) and (iv), for any split generator $G$, we have
\[\delta_t(G, (F_1F_2)^n(G))\leq \delta_t(G, F_1^n(G))\delta_t(G, F_2^n(G)).\]
Then it is easy to see that $h_t(F_1F_2)\leq h_t(F_1)+h_t(F_2)$
(\cite[Section 2.2]{DHKK}).
Since $F_1F_2$ and $F_1^{-1}$ also commute, we have 
\[
h_t(F_2)\leq h_t(F_1F_2)+h_t(F_1^{-1})=h_t(F_1F_2)+h_{-t}(F_1),
\]
where the last equality follows from Lemma \ref{lemma:Inverse}. Hence 
\[h_t(F_2)-h_{-t}(F_1)\leq h_t(F_1F_2)\leq h_t(F_1)+h_t(F_2).\] When $h_t(F_1)$ is an odd function, we get the claimed equality. In this case, 
\begin{align*}
\ph_t(F_1F_2)&=\limsup_{n\to\infty}\frac{\log\delta_t(G, (F_1F_2)^n(G))-nh_t(F_1F_2)}{\log(n)}\\
&\leq\limsup_{n\to \infty}\frac{\log\delta_t(G, F_1^n(G))+\log\delta_t(G, F_2^n(G))-nh_t(F_1)-nh_t(F_2)}{\log(n)}\\
&\leq\limsup_{n\to \infty}\frac{\log\delta_t(G, F_1^n(G))-nh_t(F_1)}{\log(n)}+\limsup_{n\to \infty}\frac{\log\delta_t(G, F_2^n(G))-nh_t(F_2)}{\log(n)}\\
&=\ph_t(F_1)+\ph_t(F_2).
\end{align*}
Using again the fact that $F_1F_2$ commutes with $F_1^{-1}$ and $h_t(F_1^{-1})=h_{-t}(F_1)$, 
 we get 
\[\ph_t(F_2)\leq \ph_t(F_1^{-1})+\ph_t(F_1F_2)=\ph_{-t}(F_1)+\ph_t(F_1F_2),\]
where the last equality uses Lemma \ref{lemma:Inverse}. Therefore, 
\[\ph_t(F_2)-\ph_{-t}(F_1)\leq \ph_t(F_1F_2)\leq \ph_t(F_1)+\ph_t(F_2).\] This gives the claimed equality when $\ph_t(F_1)$ is an odd function.
\end{proof}

\section{Polynomial mass growth rate}\label{sec:MassGrowth}
As is proposed and sketched in \cite[Section 4.5]{DHKK} and worked out in detail in \cite{Ikeda}, when the triangulated category $\cD$ admits a Bridgeland stability condition, one can measure the complexity of an endofunctor on $\cD$ by the so-called (exponential) \textit{growth rate of mass}. We study in this section the \textit{polynomial} analogue of  mass growth and compare it to the categorical polynomial entropy. Let us first recall some basic notions.

Let $\cD$ be a triangulated category and $\sigma=(Z, \mathcal{A})$ be a stability condition in the sense of Bridgeland \cite{BridgelandAnnals}, where $Z\colon K_0(\cD)\rightarrow \Gamma \to \mathbb{C}$ is a homomorphism\footnote{We always assume that $Z$ factors through some finite-rank free abelian group $\Gamma$, which is often taken to be the numerical Grothendieck group $\mathcal{N}(\cD)$ in this paper, and that $\sigma$ satisfies the support property of \cite{KontsevichSoibelman08}.} called the \textit{central charge}, and $\mathcal{A}$ is the heart of a bounded t-structure on $\cD$. Then for any non-zero object $E$, its \textit{mass function} with respect to $\sigma$ is defined as the following real function in $t$:
\[m_{\sigma, t}(E)\coloneqq\sum_k |Z(A_k)|e^{\phi(A_k)t},\]
where $A_k$ are the $\sigma$-semistable factors of $E$ and  $\phi$ is the phase function. Denote $m_\sigma\coloneqq m_{\sigma,0}$. The space of stability conditions is denoted by $\Stab(\cD)$ (or more precisely $\Stab_\Gamma(\cD)$ if one want to specify the choice of $\Gamma$), which is naturally a complex manifold of dimension $\operatorname{rk}(\Gamma)$, by \cite{BridgelandAnnals}.

\begin{rmk}\label{rmk:PropOfMass}
We collect some fundamental  properties of the mass function, due to \cite{DHKK} and \cite{Ikeda}.
\begin{enumerate}[label=(\roman*)]
\item (Triangle inequality). For any distinguished triangle $E'\to E\to E''\xrightarrow{+1}$, we have
\[m_{\sigma, t}(E)\leq m_{\sigma, t}(E')+m_{\sigma, t}(E'').\]
\item For any non-zero objects $E, E'$,  $m_{\sigma, t}(E)\leq m_{\sigma, t}(E')\delta_t(E', E)$, where $\delta_t$ is the complexity function in Definition \ref{def:Complexity}.
\item If the distance (defined in \cite{BridgelandAnnals}) between two stability conditions $\sigma, \tau$ is finite, then there exist two functions $C_1(t), C_2(t)\colon \mathbb{R}\to \mathbb{R}_{>0}$, such that for any non-zero object $E$, we have
\[C_1(t) m_{\tau, t}(E)< m_{\sigma, t}(E)< C_2(t) m_{\tau, t}(E).\]
\end{enumerate}
\end{rmk}

Recall also the definition of the mass growth in \cite[Section 4.5]{DHKK}.
\begin{defn}[Mass growth]\label{def:MassGrowth}
Let $\cD$ be a triangulated category endowed with a stability condition $\sigma$. Let $G$ be a split generator of $\cD$, then the \textit{mass growth function} of an endofunctor $F$ is defined as 
\[
h_{\sigma,t}(F) \coloneqq
\limsup_{n\ra\infty}
\frac{\log m_{\sigma,t}(F^n(G))}{n}.
\]
It is shown in \cite[Theorem 3.5]{Ikeda} that the definition is independent of the choice of the split generator, and 
\[
h_{\sigma,t}(F) = \sup_{0\neq E\in\cD}
\Big\{\limsup_{n\ra\infty}
\frac{\log m_{\sigma,t}(F^n(E))}{n}
\Big\}.
\]
Moreover, it depends only on the connected component of $\Stab(\cD)$ in which $\sigma$ lies \cite[Proposition 3.10]{Ikeda}.
We denote $h_{\sigma}(F)\coloneqq h_{\sigma, 0}(F)$, the value at $t=0$.
\end{defn}

We propose the following notion as a secondary measurement of the mass growth, similarly to Definition \ref{def:CatPolEntropy}.
\begin{defn}[Polynomial mass growth]
In the same setting as in Definition \ref{def:MassGrowth}, the \textit{polynomial mass growth function} of $F$ is defined to be
\[\ph_{\sigma,t}(F) \coloneqq \limsup_{n\ra\infty}
\frac{\log m_{\sigma,t}(F^n(G))-nh_{\sigma,t}(F)}{\log (n)},\]
where $h_{\sigma, t}(F)$ is the mass growth of $F$. Let $\ph_{\sigma}(F)\coloneqq \ph_{\sigma, 0}(F)$.
\end{defn}

\begin{lemma}
The definition of  $\ph_{\sigma, t}(F)$ is independent of the choice of the split generator $G$. Moreover, 
\[
\ph_{\sigma,t}(F) = \sup_{0\neq E\in\cD}
\Big\{\limsup_{n\ra\infty}
\frac{\log m_{\sigma,t}(F^n(E))-nh_{\sigma,t}(F)}{\log (n)}
\Big\}.
\]
\end{lemma}
\begin{proof}
The proof is the same as in \cite[Theorem 3.5(1)]{Ikeda}, by using Remark \ref{rmk:PropOfMass} (ii).
\end{proof}

\begin{lemma}[Deforming stability conditions]\label{lemma:DeformStab}
The function $\ph_{\sigma, t}(F)$ only depends on the connected component of $\Stab(\cD)$ in which $\sigma$ lies. 
\end{lemma}
\begin{proof}
Similarly as in \cite[Proposition 3.10]{Ikeda}, it is a direct consequence of Remark \ref{rmk:PropOfMass} (iii).
\end{proof}

Given an endofunctor, to relate its mass growth to its entropy, Ikeda \cite[Theorem 3.5 (2)]{Ikeda} showed that $h_{\sigma,t}(F)\leq h_t(F)$. The following is its polynomial counterpart.
\begin{lemma}[Comparison with polynomial entropy]\label{lemma:ComparingEntropyMass}
For any real number $t$,
if $h_{\sigma,t}(F)=h_t(F)$, then 
\[\ph_{\sigma,t}(F)\leq \ph_t(F).\]
\end{lemma}
\begin{proof}
Fix a split generator $G$.  Remark \ref{rmk:PropOfMass}(ii) implies that
\[m_{\sigma, t}(F^n(G))\leq m_{\sigma, t}(G)\delta_t(G, F^n(G)).\] Hence
\[\log m_{\sigma, t}(F^n(G))-n h_{\sigma,t}(F)\leq \log m_{\sigma, t}(G)+\log \delta_t(G, F^n(G))-nh_t(F).\] One can conclude by dividing by $\log(n)$ and taking limit.
\end{proof}

Recall that a stability condition $\sigma=(Z, \mathcal{A})$ is called \emph{algebraic} if the corresponding heart $\mathcal{A}$ is a finite length abelian category with finitely many isomorphism classes of simple objects. Examples of triangulated categories admitting algebraic stability conditions include derived categories with full strong exceptional collection, derived categories of (homologically) finite-dimensional dg-modules over a connective dg-algebra of finite type etc. 

\begin{lemma}[Algebraic stability conditions]
\label{lemma:algstabilitycond}
If a connected component $\Stab^\circ(\cD)\subset\Stab(\cD)$ contains an algebraic stability condition, then for any $\sigma\in\Stab^\circ(\cD)$, we have
\[
h^\pol_{\sigma,t}(F)=h^\pol_t(F).
\]
\end{lemma}

\begin{proof}
The proof is similar to \cite[Theorem~3.14]{Ikeda}. Ikeda proved $h_{\sigma,t}(F)=h_t(F)$  by showing more strongly that there exists a special algebraic stability condition $\sigma_0$ in the same connected component of the stability manifold, such that 
\[e^{\frac{t}{2}}\delta_t(G, F^n(G))\leq m_{\sigma_0, t}(F^n(G))\leq m_{\sigma_0, t}(G)\delta_t(G, F^n(G)),\] which allows us to deduce that $h^\pol_{\sigma_0,t}(F)=h^\pol_t(F)$. One concludes by Lemma \ref{lemma:DeformStab}.
\end{proof}

\section{Yomdin-type Estimates}
In the spirit of Gromov--Yomdin's Theorem \ref{thm:GromovYomdin}, given an endofunctor of a triangulated category, we want to understand its polynomial entropy and its polynomial mass growth rate, which are of categorical nature, in terms of some cohomological data, which is essentially a matter of linear algebra.
\subsection{Polynomial growth rate in linear algebra}
Let us recall some linear algebra facts here. Given a square complex matrix $M$, let $\rho(M)$ denote the \textit{spectral radius} of $M$, namely, the maximal absolute value of the eigenvalues of $M$. If $M$ is a virtually unipotent\footnote{A matrix is called \textit{virtually unipotent}, or \textit{quasi-unipotent}, if a positive power of the matrix is unipotent.} matrix (so $\rho(M)=1$), then the growth of $\|M^n\|$ is asymptotically $n^{s(M)}$, where $s(M)+1$ is the size of the maximal Jordan block of $M$; the integer $s(M)$ is called the \textit{polynomial growth rate} of $M$ in this case.
If $\rho(M)>1$, then the growth of $\|M^n\|$, when $n\to \infty$, is asymptotically $\rho(M)^nn^{s(M)}$, where $s(M)+1$ is the size of maximal Jordan blocks with eigenvalues having maximal modulus; we feel that it is meaningful to call $s(M)$ the polynomial growth rate of $M$. To make this idea precise, we propose the following definition, which generalizes the notion of polynomial growth rate used in \cite[Section 2]{CanPR}, by normalizing the exponential growth rate determined by the spectral radius. 
\begin{defn}[Polynomial growth rate]
\label{def:polygrowthratelinearmap}
Let $\phi$ be an endomorphism of a finite-dimensional vector space endowed with some norm $\|-\|$. The \textit{polynomial  growth rate} of $\phi$ is defined to be
\[
s(\phi)\coloneqq\lim_{n\ra\infty}\frac{\log\|\phi^n\|-n\log(\rho(\phi))}{\log (n)}.
\]
As all norms on the space of matrices are equivalent, $s(\phi)$ is independent of the choice of the norm.
\end{defn}
Let us record the following basic result.
\begin{lemma}\label{lemma:PolGrowLim}
Notation is as before. The limit in Definition \ref{def:polygrowthratelinearmap} exists, and it is precisely one less than the maximal size of the Jordan blocks whose eigenvalues are of maximal absolute value $\rho(\phi)$. In particular, $s(\phi)$ is a natural number. 
\end{lemma}
\begin{proof}
Let $\frac{\phi}{\rho(\phi)}=D+N$ be the Jordan decomposition, where $D$ is semisimple, $N$ is nilpotent and $ND=DN$. Then the eigenvalues of $D$ are of modulus $\leq 1$. Let $s+1$ be the maximal size of the Jordan blocks of $\frac{\phi}{\rho(\phi)}$ whose eigenvalues are of maximal modulus 1. We have 
\begin{equation*}
\frac{\phi^n}{\rho(\phi)^n}=\sum_{j<s} {{n}\choose{j}}D^{n-j}N^j+{{n}\choose{s}}D^{n-s}N^s+\sum_{j>s} {{n}\choose{j}}D^{n-j}N^j,
\end{equation*}
where, on the right-hand side, the norm of the first term has growth at most $O(n^{s-1})$, the norm of the second term has growth equivalent to $n^s$, and the third term tends to zero, when $n\to \infty$. Therefore, $\left\|\frac{\phi^n}{\rho(\phi)^n}\right\|$ has growth equivalent to $n^s$.
\end{proof}

\subsection{Lower bound for categorical polynomial entropy}
Given a saturated triangulated category $\cD$, its \textit{numerical Grothendieck group}, denoted by $\cN(\cD)$, is by definition the quotient of the Grothendieck group $K_0(\cD)$ by the radical of the Euler pairing $\chi(E, E')\coloneqq \sum_{k\in \mathbb{Z}}(-1)^k\dim \Hom(E, E'[k])$. 

We establish the following Yomdin-type lower bound for the categorical polynomial entropy in terms of the induced action on the numerical Grothendieck group. In passing, we provide an alternative proof for the lower-bound of the categorical entropy previously obtained by \cite[Theorem 2.13]{KST18}.

\begin{prop}[Yomdin-type lower bound]\label{prop:GYLowerBoundPolEntropy}
Let $\cD$ be a saturated triangulated category with a split generator $G$. Let $F$ be an endofunctor of $\cD$. Denote by $\cN(F)$ the induced endomorphism of the numerical Grothendieck group $\cN(\cD)$. Then we have (\cite{KST18})
\[h^\cat(F)\geq \log \rho(\mathcal{N}(F)).\]
If the equality holds (for example when $h^\cat(F)=0$), then
\[ h^\pol(F)\geq s(\cN(F)),\]
where $s$ is the polynomial growth rate.
\end{prop}
\begin{proof}
For ease of notation, denote  $f\coloneqq \mathcal{N}(F)\in \operatorname{End}(\mathcal{N}(\cD))$.
Let $\lambda$ be an eigenvalue of $f$ with $|\lambda|=\rho(f)$ such that its characteristic space $\ker(f-\lambda \operatorname{id})^\infty$ has a maximal Jordan block, whose size is denoted  $s+1$ ($s\geq 0$). Let $v\in \mathcal{N}(\cD)_\mathbb{C}$ be a vector such that 
$\{v=v_0, v_1, \dots, v_s\}$ is a basis of such a maximal Jordan  block, where $v_k\coloneqq(f-\lambda\operatorname{id})^kv$ for $k=0, \dots, s$.

Take objects $M_1, \dots, M_m$ in $\cD$ such that their classes in $\cN(\cD)$ form a basis. We define the following norm on $\mathcal{N}(\cD)_{\mathbb{C}}$:
\[ \| w\|\coloneqq\sum_{i=1}^{m} |\chi_{\mathbb{C}}([M_i], w)| \]
for any $w \in \cN(\cD)_\mathbb{C}$, where $\chi_\mathbb{C}$ is the linear extension of the Euler pairing $\chi$.

Write $v=\sum_{i=1}^ma_i[M_i]$ with $a_i\in \mathbb{C}$. Choose positive  integers $l_i> |a_i|$ for all $i=1, \dots, m$. Consider the object $E=\bigoplus_{i=1}^m M_i^{\oplus l_i}$. We have for any $n>0$,
\begin{align*}
    \epsilon\left(G\oplus \bigoplus_{i=1}^mM_i, F^n(G\oplus E)\right)
    &\geq \epsilon\left(\bigoplus_{i=1}^mM_i, F^n(E)\right)\\
    &= \sum_{i=1}^m\sum_{j=1}^m l_j \epsilon(M_i, F^n(M_j))\\
    &\geq \sum_{i=1}^m\sum_{j=1}^m |a_j| \epsilon(M_i, F^n(M_j))\\
    &\geq \sum_{i=1}^m\sum_{j=1}^m |a_j|\cdot |\chi(M_i, F^n(M_j))|\\
    &=\sum_{j=1}^m|a_j|\cdot \|F^n(M_j)\|\\
    &\geq \|f^n(v)\|\\
    &= \|\lambda^nv_0+\binom{n}{1}\lambda^{n-1}v_1+\cdots+\binom{n}{s}\lambda^{n-s}v_s\|.
\end{align*}
Since $G\oplus \bigoplus_{i=1}^mM_i$ and  $G\oplus E$ are split generators, 
we obtain 
\begin{align*}
h^\cat(F)&=\lim_{n\to \infty}\frac{\log\epsilon\left(G\oplus \bigoplus_{i=1}^mM_i, F^n(G\oplus E)\right)}{n}\\
&\geq \lim_{n\to \infty}\frac{\log\|\lambda^nv_0+\binom{n}{1}\lambda^{n-1}v_1+\cdots+\binom{n}{s}\lambda^{n-s}v_s\|}{n}\\&=\log|\lambda|\\&=\log \rho(\mathcal{N}(F)). 
\end{align*}
This recovers \cite[Theorem 2.13]{KST18}.\\
Now if $h^\cat(F)=\log|\lambda|$, by Lemma \ref{lemm:EntropyInExt-distance}, we have 
\begin{align*}
h^\pol(F)&=\limsup_{n\to \infty}\frac{\log\epsilon\left(G\oplus \bigoplus_{i=1}^mM_i, F^n(G\oplus E)\right)-nh^\cat(F)}{\log(n)}\\
&\geq \limsup_{n\to \infty}\frac{\log\|\lambda^nv_0+\binom{n}{1}\lambda^{n-1}v_1+\cdots+\binom{n}{s}\lambda^{n-s}v_s\|-n\log|\lambda|}{\log(n)}\\&=s\\&=s(\mathcal{N}(F)). \qedhere
\end{align*}
\end{proof}



In general, the inequality in Proposition \ref{prop:GYLowerBoundPolEntropy} can be strict, see Example \ref{ex:GYCouldBeStrict}. We give here an example where the previously established lower bound is achieved. More examples will be presented in Section \ref{sec: Examples}. Recall that an associative algebra is called \textit{hereditary} if its has global dimension at most 1. Important examples of hereditary algebras include semisimple algebras, path algebras of finite quivers without oriented cycles etc. 
\begin{prop}[Hereditary algebras]
\label{prop:hereditaryalgebraGY}
Let $A$ be a hereditary finite dimensional (not necessarily commutative) $\bC$-algebra. 
Then for any autoequivalence $F$ of $\cD^b(A)$, we have 
$h^\pol(F)=s(\cN(F))$.
\end{prop}

\begin{proof}
There are projective $A$-modules $P_1, \cdot \cdot \cdot, P_d$ such that $\left<P_1, \cdot \cdot \cdot, P_d\right>$ is a full strong exceptional collection of $\cD^b(A)$.
Let $v_i\coloneqq [P_i] \in \cN(A)$ for $1 \leq i \leq d$. Then $v_1, \cdot \cdot \cdot, v_d$ is a basis of the numerical Grothendieck group $\cN(A)$ of $D^b(A)$.
Consider the following norm on $\cN(A)_\mathbb{R}$:
$$\|w\|\coloneqq \sum_{i=1}^d |\chi_\mathbb{R}(v_i, w)|.$$
Since $A$ is hereditary, an indecomposable object of $\cD^b(A)$ is isomorphic to a shift of an indecomposable $A$-module.
Note that $F^n(P_i)$ is indecomposable for all $n \geq 1$. Therefore, we have 
\[\epsilon(P_i, F^n(P_j))=|\chi(P_i, F^n(P_j))|\]
for $1 \leq i, j \leq d$. 
By \cite[Proposition 2.14]{KST18}, 
\[ h^\cat(F)=\log \rho(\cN(F)). \]
Let $M:=\max\{\|v_i\| \mid 1 \leq i \leq d \}$.
Note that \[\|\cN(F)^n v\| \leq \|\cN(F)^n\|\cdot\|v\|\]
for any $v\in \cN(A)$.
Therefore, we have 
\begin{align*}
h^\pol(F) &= \limsup_{n\ra\infty} \frac{\log\epsilon(\bigoplus_{i=1}^{d}P_i,F^n(\bigoplus_{j=1}^{d}P_j))-n\log\rho(\cN(F))}{\log (n)} \\
&= \limsup_{n\ra\infty} \frac{\log \sum_{i, j=1}^{d}\epsilon(P_i, F^n(P_j))-n\log\rho(\cN(F))}{\log (n)} \\
&= \limsup_{n\ra\infty} \frac{\log \sum_{i,j=1}^{m} |\chi_{\bR}(v_i, \cN(F)^nv_j)|-n\log\rho(\cN(F))}{\log (n)} \\
&=  \limsup_{n\ra\infty} \frac{\log \sum_{j=1}^{d}\|\cN(F)^nv_j \|-n\log\rho(\cN(F))}{\log (n)}\\
&\leq \limsup_{n\ra\infty} \frac{\log \|\cN(F)^n \|+\log d+\log M-n\log\rho(\cN(F))}{\log (n)}\\
&=\limsup_{n\ra\infty} \frac{\log \|\cN(F)^n \|-n\log\rho(\cN(F))}{\log (n)}\\
&=s(\cN(F)).
\end{align*}
One can deduce the desired equality by combining it with Proposition \ref{prop:GYLowerBoundPolEntropy}.
\end{proof}

\subsection{Lower bound for polynomial mass growth}
We establish the analogue of Proposition \ref{prop:GYLowerBoundPolEntropy} for the polynomial mass growth rate (see Section \ref{sec:MassGrowth}) in the presence of Bridgeland stability conditions.

\begin{prop}[Yomdin-type lower bound]\label{prop:GYLowerBoundMassGrowth}
Let $\cD$ be a triangulated category with a split generator $G$. Assume $\cD$ admits a stability condition $\sigma$ that factors through $\cN(\cD)$,  the numerical Grothendieck group of $\cD$. Let $F$ be an endofunctor of $\cD$. Denote  $\cN(F)$ the induced endomorphism on $\cN(\cD)$. If we have 
$h_\sigma(F)= \log \rho(\mathcal{N}(F))$
(for example when $h_\sigma(F)=0$), then
\[ h^\pol_\sigma(F)\geq s(\cN(F)),\]
where $s$ is the polynomial growth rate.\\
More generally, if the central charge of the stability condition $\sigma$ factors through $\mathrm{cl}\colon K_0(\cD)\rightarrow\Gamma$ for some lattice $\Gamma$, and suppose that the homomorphism $\mathrm{cl}$ is surjective and its kernel is preserved by $F$.
Then we have $h^\pol_\sigma(F)\geq s(\Gamma(F))$ provided that $h_\sigma(F)= \rho(\Gamma(F))$, where $\Gamma(F)$ is the induced endomorphism on $\Gamma$.
\end{prop}
\begin{proof}
We only prove the case where the central charge $Z$ factors through $\mathcal{N}(\cD)$, the general case is similar. We proceed as in the proof of Proposition \ref{prop:GYLowerBoundPolEntropy}. Keeping the same notation there and using the fact that $m_\sigma(-)\geq |Z(-)|$, we have that for any $n>0$,
\begin{align*}
    m_\sigma(F^n(G\oplus E))
    &\geq m_\sigma(F^n(E))\\
    &= \sum_{j=1}^m l_j m_\sigma(F^n(M_j))\\
    &\geq \sum_{j=1}^m |a_j|   m_\sigma(F^n(M_j))\\
    &\geq \sum_{j=1}^m |a_j|\cdot | Z(F^n(M_j))|\\
    &=\sum_{j=1}^m|a_j|\cdot |Z(f^n([M_j]))|\\
    &\geq |Z(f^n(v))|\\
    &= \left|Z\left(\lambda^nv_0+\binom{n}{1}\lambda^{n-1}v_1+\cdots+\binom{n}{s}\lambda^{n-s}v_s\right)\right|.
\end{align*}
If $h_\sigma(F)=\log|\lambda|$, assume moreover that $Z(v_s)\neq 0$, since $G\oplus E$ is a split generator, we have the following, by Lemma \ref{lemm:EntropyInExt-distance},
\begin{align*}
h^\pol_\sigma(F)&=\limsup_{n\to \infty}\frac{\log m_\sigma(F^n(G\oplus E))-nh_\sigma(F)}{\log(n)}\\
&\geq \limsup_{n\to \infty}\frac{\log\left|Z\left(\lambda^nv_0+\binom{n}{1}\lambda^{n-1}v_1+\cdots+\binom{n}{s}\lambda^{n-s}v_s\right)\right|-n\log|\lambda|}{\log(n)}\\&=s\\&=s(\mathcal{N}(F)). 
\end{align*}
The condition $Z(v_s)\neq 0$ can always be achieved by deforming the stability condition $\sigma$ (note that $\dim \Stab(\cD)=\dim \mathcal{N}(F)$), and this does not affect $h_\sigma(F)$ or $\ph_\sigma(F)$, thanks  to Lemma \ref{lemma:DeformStab}.
\end{proof}

\begin{cor}\label{cor:hereditaryalgebraGYMass}
In the situation of Proposition \ref{prop:hereditaryalgebraGY}, if moreover $\cD$ admits a numerical stability condition $\sigma$, then we have \[h^\pol_{\sigma}(F)=s(\cN(F)).\]
\end{cor} 
\begin{proof}
By \cite[Proposition 2.14]{KST18} and \cite[Theorem 1.2]{Ikeda}, we have that 
\[ h^\cat(F)=h_{\sigma}(F)=\log \rho(\cN(F)). \]
Therefore, Lemma \ref{lemma:ComparingEntropyMass} and Proposition \ref{prop:GYLowerBoundMassGrowth} imply that 
\[\ph(F)\geq \ph_{\sigma}(F)\geq s(\mathcal{N}(F)).\] Then Proposition \ref{prop:hereditaryalgebraGY} allows us to conclude the proof.
\end{proof}

\section{Classical dynamical systems: a categorical retake}
In this section, we make a connection of the categorical theory developed so far to the classical setting.
Let $\k$ be an algebraically closed field and let $X$ be a smooth projective $\k$-variety endowed with a surjective (regular) endomorphism $f\colon X\to X$. Note that $f$ is automatically finite and flat. There has been extensive study of the complexity of the system $(X, f)$ by topological, geometric, algebraic, analytic and even probabilistic approaches. We employ here a categorical method (as in \cite{DHKK}, \cite{KiTa}) by looking at the naturally associated categorical dynamical system $(\cD^b(X), \mathbb{L}f^*)$.

\subsection{Dynamical degrees}\label{sec:DynamicalDegree}
We recall some basic properties of a series of fundamental invariants, called dynamical degrees, of an algebraic/complex dynamical system. The material here is well-established in the literature.  
We are in the following more broad setting:
\begin{itemize}
    \item $f$ is a dominant rational self-map of a normal projective variety $X$ defined over $\k$, or 
    \item when $\k=\mathbb{C}$, $f$ is a dominant meromorphic self-map of a compact K\"ahler manifold $X$.
\end{itemize}

Let $f\colon X\dashrightarrow X$ be as above. Denote by $d$ the dimension of $X$. For any integer $0\leq p\leq d$, the $p$-th \textit{dynamical degree} of $f$, denoted by $d_p(f)$, is by definition
\[d_p(f)\coloneqq \lim_{n\to \infty} \left(\int_X (f^n)^*\omega^p\wedge \omega^{d-p}\right)^{\frac{1}{n}},\]
where $\omega$ is the first Chern class of an ample line bundle (or a K\"ahler class in the compact K\"ahler setting). The definition is independent of the choice of $\omega$ and it is invariant under birational conjugation. The existence of the limit is due to Dinh--Sibony \cite{DinhSibonyAnnals, DinhSibonyENS} when $\k=\mathbb{C}$, and to Truong \cite{TTTruong} (see Dang \cite{Dang} for an alternative treatment) when $\operatorname{char}(\k)$ is arbitrary.

Let $N^p\coloneqq N^p(X)$ be the group of algebraic cycles of codimension $p$ on $X$ modulo numerical equivalence. For a rational self-map $g$ of $X$, we denote by $g^*_{N^p}$ the induced endomorphism on $N^p(X)$. Thanks to \cite[Theorem 2]{Dang}, we have
\[d_p(f)=\lim_{n\to \infty}\|(f^n)^*_{N^p}\|^{\frac{1}{n}},\]  which is equal to $\rho(f^*_{N^p})$ when $f$ is regular, where $\rho$ denotes the spectral radius.

Similarly, in the compact K\"ahler setting, we have 
\[d_p(f)=\lim_{n\to \infty}\|(f^n)^*_{H^{p,p}}\|^{\frac{1}{n}},\]  which is equal to $\rho(f^*_{H^{p,p}})$ when $f$ is holomorphic (or more generally, algebraically stable), where $H^{p,p}$ is the $(p,p)$-part in the Hodge decomposition of $H^{2p}(X, \mathbb{C})$.

By the Teisser--Khovanskii inequality (see \cite{GromovConvexSets}), the sequence $\{d_p(f)\}_{p=0}^d$ is log-concave (\cite{DinhSibonyENS, TTTruong, Dang}). When $\k=\mathbb{C}$, a celebrated theorem of Gromov--Yomdin \cite{Gromov, Yomdin} (cf.~Theorem \ref{thm:GromovYomdin}) says that when $f$ is holomorphic (and surjective), the topological entropy can be computed from the dynamical degrees as well as the action on cohomology:
$$h_{\operatorname{top}}(f)=\max_{0\leq p\leq d} \log(d_p(f))=\log \lim_{n\to \infty}\|(f^n)^*_{H^*}\|^{\frac{1}{n}}=\log \rho(f^*_{H^*}).$$
As a consequence, $f$ has positive topological entropy if and only if $d_1(f)>1$.

\subsection{Polynomial dynamical degrees}\label{sec:PolDynamicalDegree}
Keep  the same setting as in Section \ref{sec:DynamicalDegree}. Consider $f\colon X\dashrightarrow X$ as before.
In the K\"ahler situation, under the hypothesis that the topological entropy of $f$ is zero, Lo Bianco \cite[Section 1.3]{LoB} initiated the study of the so-called \textit{polynomial dynamical degrees}, the polynomial counterpart of dynamical degrees discussed in Section \ref{sec:DynamicalDegree}, for holomorphic endomorphisms, which was later extended by Cantat--Paris-Romaskevich \cite[Section 3]{CanPR} to meromorphic self-maps. 

Using the idea of normalizing the exponential growth as in Definition \ref{def:polygrowthratelinearmap}, we generalize this notion of polynomial dynamical degrees to rational/meromorphic self-maps of arbitrary entropy over arbitrary algebraically closed base field $\k$:
\begin{defn}[Polynomial dynamical degree]\label{def:PolDynDeg}
Let $f\colon X\dashrightarrow X$ be as above.
For any integer $0\leq p\leq d$, the $p$-th \textit{polynomial dynamical degree}, denoted by $s_p(f)$, is by definition
\[s_p(f)\coloneqq \limsup_{n\to \infty}\frac{\log(\int_X (f^n)^*\omega^p\wedge \omega^{d-p})-n\log d_p(f)}{\log (n)},\]
where $\omega$ is the first Chern class of an ample line bundle (or a K\"ahler class when $\k=\mathbb{C}$) and $d_p(f)$ is the $p$-th dynamical degree of $f$. We will see shortly in Proposition \ref{prop:PolDynDegLimit} that if $f$ is regular/holomorphic, then the $\limsup$ is actually a limit, and takes values in $\mathbb{N}$, the set of non-negative integers.
\end{defn}
\begin{rmk}
By \cite[Theorem 1 (ii)]{Dang}, the definition is independent of the choice of $\omega$, and it is invariant under birational conjugation. Thanks to \cite[Theorem 2]{Dang}, for any $0\leq p\leq d$,  we have \[s_p(f)=\limsup_{n\to \infty}\frac{\log\|(f^n)^*_{N^p}\|-n\log d_p(f)}{\log (n)},\]
where $N^p\coloneqq N^p(X)$ is the group of codimension-$p$ cycles modulo numerical equivalence.
Similarly, in the compact K\"ahler setting, 
\[s_p(f)=\limsup_{n\to \infty}\frac{\log\|(f^n)^*_{H^{p,p}}\|-n\log d_p(f)}{\log (n)}.\]
We do not know whether the $\limsup$ is a limit in general.
\end{rmk}

Recall the notion of polynomial growth rate in linear algebra in Definition \ref{def:polygrowthratelinearmap}.
\begin{prop}\label{prop:PolDynDegLimit}
Notation is as above. If $f$ is moreover regular, then for any $0\leq p\leq d$, the $\limsup$ in Definition \ref{def:PolDynDeg} is a limit, and 
\[s_p(f)=s(f^*_{N^p})=\lim_{n\to \infty}\frac{\log\|(f^n)^*_{N^p}\|-n\log d_p(f)}{\log (n)} \in \mathbb{N}.\]
Similarly, in the compact K\"ahler setting, if $f$ is moreover holomorphic, then 
\[s_p(f)=s(f^*_{H^{p,p}})=\lim_{n\to \infty}\frac{\log\|(f^n)^*_{H^{p,p}}\|-n\log d_p(f)}{\log (n)} \in \mathbb{N}.\]
\end{prop}
\begin{proof}
By \cite[Theorem 2]{Dang}, $|\log(\int_X (f^n)^*\omega^p\wedge \omega^{d-p})-\log\|(f^n)^*_{N^p}\||$ is bounded by some universal constant depending only on $X$. Since the sequence $$\frac{\log\|(f^n)^*_{N^p}\|-n\log d_p(f)}{\log (n)}=\frac{\log\|(f^*_{N^p})^n\|-n\log \rho(f^*_{N^p})}{\log (n)}$$ is convergent to a natural number by Lemma \ref{lemma:PolGrowLim}, the sequence in  Definition \ref{def:PolDynDeg} is also convergent with the same limit.
\end{proof}

In the sequel when discussing the concavity properties, we have to stay in the following more restrictive setting:
\begin{itemize}
    \item $f$ is a surjective (regular) endomorphism of a smooth projective variety $X$ defined over $\k$, or
    \item $f$ is a surjective holomorphic endomorphism of a compact K\"ahler manifold $X$.
\end{itemize}

Thanks to the log-concavity of the sequence $\{d_p(f)\}$, there exist integers $0\leq p_0\leq p_1\leq d$, such that $d_0<\dots <d_{p_0-1}< d_{p_0}=\dots=d_{p_1}>d_{p_1+1}>\dots $. The following results extend Lo Bianco's observation \cite[Proposition 1.3.9]{LoB}.
\begin{lemma}[Concavity]\label{lemma:ConcavityS}
Let $f\colon X\to X$ be as above.
The sequence $\{s_p(f)\}_{p=p_0}^{p_1}$ is concave.
\end{lemma}
\begin{proof}
Let $\lambda=d_{p_0}(f)=\dots=d_{p_1}(f)$.
By the Teisser--Khovanskii inequality (\cite{GromovConvexSets}), for any $n>0$, $$p\mapsto \log(\int_X (f^n)^*\omega^p\wedge \omega^{d-p})-n \log d_p(f)=\log(\int_X (f^n)^*\omega^p\wedge \omega^{d-p})-n \log \lambda$$
is a concave sequence for $p_0\leq p\leq p_1$. One concludes by dividing by $\log(n)$ and taking limit. Note that the argument does not apply to the more general case where $f$ is only assumed to be rational/meromorphic, as $\limsup$ does not preserve the concavity. 
\end{proof}

\begin{prop}\label{prop:PolDynDegCohomology}
Notation is as before (in particular, $f$ is regular). Let $N^*\coloneqq \bigoplus_pN^p(X)$. We have 
\[s(f^*_{N^*})=\max_{p_0\leq p \leq p_1}s_p(f).\]
Similarly, in the compact K\"ahler setting, 
\[s(f^*_{H^*})=\max_{p_0\leq p \leq p_1}s_p(f).\]
\end{prop}
\begin{proof}
Let $\lambda=\max_p d_p(f)$.
The case for $N^*$ is almost immediate:
\[s(f^*_{N^*})=\lim_{n\to \infty}\frac{\log\|(f^*_{N^*})^n\|-n\log \lambda}{\log n}=\lim_{n\to \infty}\frac{\log\|(f^*_{\oplus_{p_0\leq p\leq p_1} N^p})^n\|-n\log \lambda}{\log n}=\max_{p_0\leq p \leq p_1}s(f^*_{N^p})=\max_{p_0\leq p \leq p_1}s_p(f).\]
In the compact K\"ahler  situation, we need to show in addition that for any $0\leq i, j\leq d$, if $\rho(f^*_{H^{i, j}})=\lambda$, then $s(f^*_{H^{i, j}})\leq \max_{p_0\leq p \leq p_1}s_p(f)$. Given such a couple $(i,j)$ with $\rho(f^*_{H^{i, j}})=\lambda$, as it is shown in \cite[Proposition 5.8]{DinhJGA} that $\rho(f^*_{H^{i,j}})\leq \sqrt{d_i(f)d_j(f)}$, we must have that $p_0\leq i, j\leq p_1$. Let $(f, f)$ be the endomorphism of $X\times X$  sending $(x, x')$ to $(f(x), f(x'))$, then it is easy to see (cf.~loc.~cit.) that there exists a constant $C>0$ such that for any $n>0$,
\begin{equation}\label{eqn:ffControlf}
    \|(f^n)^*_{H^{i,j}}\|^2\leq C\|(f^n, f^n)^*_{H^{i+j, i+j}}\|.
\end{equation}

Since $\rho(f^*_{H^{i,j}})=\lambda$ by assumption, \eqref{eqn:ffControlf} implies that 
\[d_{i+j}(f,f)\geq \lambda^2.\]
Let us show that this is actually an equality and compute the corresponding polynomial dynamical degree of $(f,f)$.
For any K\"ahler form $\omega$ on $X$, $\pi_1^*\omega+\pi_2^*\omega$ is a K\"ahler form on $X\times X$, where $\pi_1$ and $\pi_2$ are the natural projections from $X\times X$ to $X$. Let $C_{i,j,l}\coloneqq  {{i+j}\choose{l}}{{2d-i-j}\choose{2d-l}}$, then 
\begin{align*}
    &\int_{X\times X}(f^n,f^n)^*(\pi_1^*\omega+\pi_2^*\omega)^{i+j}\wedge (\pi_1^*\omega+\pi_2^*\omega)^{2d-i-j}\\
    =& \sum_{l=0}^{i+j}C_{i,j,l}\left(\int_X(f^n)^*\omega^l\wedge\omega^{d-l}\cdot\int_X(f^n)^*\omega^{i+j-l}\wedge\omega^{d+l-i-j}\right),
\end{align*}
where the $l$-th term has growth in $n$ equivalent to $d_l(f)^n n^{s_l(f)} d_{i+j-l}(f)^n n^{s_{i+j-l}(f)}$, hence the sum has growth equivalent to $\lambda^{2n} n^{s}$,
with $$s=\max_{\substack{p_0\leq l\leq p_1\\p_0\leq i+j-l\leq p_1}}(s_l(f)+s_{i+j-l}(f)).$$
Therefore $d_{i+j}(f,f)=\lambda^2$ and 
\[s_{i+j}(f,f)= \max_{\substack{p_0\leq l\leq p_1\\p_0\leq i+j-l\leq p_1}}(s_l(f)+s_{i+j-l}(f)).\]
By the concavity of the sequence $s_{p_0}, \dots, s_{p_1}$ (Lemma \ref{lemma:ConcavityS}), if $i+j=2p$ is even, then $s_{i+j}(f,f)=2s_p(f)$; if $i+j=2p+1$ is odd, then $s_{i+j}(f,f)=s_p(f)+s_{p+1}(f)$. In any case, $s_{i+j}(f)\leq 2\max_{p_0\leq p \leq p_1}s_p(f)$.

Now use again \eqref{eqn:ffControlf}, we obtain that 
\begin{align*}
s(f^*_{H^{i,j}})
&=\lim_{n\to \infty}\frac{\log \|(f^n)^*_{H^{i,j}}\|- n\log(\lambda)}{\log(n)}\\
&\leq \lim_{n\to \infty}\frac{\log(C)+\log \|(f^n, f^n)^*_{H^{i+j, i+j}}\| - n\log(\lambda^2)}{2\log(n)}\\
&= \frac{1}{2}s_{i+j}(f,f)\\
&\leq \max_{p_0\leq p \leq p_1}s_p(f).\qedhere
\end{align*}
\end{proof}

\subsection{Using derived categories}
Let $f\colon X\to X$ be a surjective endomorphism of a smooth projective variety $X$ defined over $\k$.
Let $\cD^b(X)$ be the bounded derived  category of coherent sheaves on $X$.  Let  $\mathbb{L}f^* \colon \cD^b(X)\to \cD^b(X)$ be the derived pullback functor. As $f$ is flat, $\mathbb{L}f^*=f^*$, but we will continue writing $\mathbb{L}f^*$ in the sequel to remind that it is a functor and to avoid confusion with the induced action on cohomology.  When  $\k=\mathbb{C}$, Kikuta--Takahashi \cite{KiTa} showed that its categorical entropy coincides with the topological entropy of $f$:
\begin{equation}
    h^\cat(\mathbb{L}f^*)=h_{\operatorname{top}}(f).
\end{equation}
This result was later extended over any algebraically closed field by Ouchi \cite[Theorem 5.2]{Ouchi}. Let us state their results as follows, incorporating also the discussion in Section \ref{sec:DynamicalDegree} as well as some contribution from Ikeda \cite{Ikeda}.
\begin{thm}[Kikuta--Takahashi, Ouchi]\label{thm:KTO}
Let $f\colon X\to X$ be a surjective endomorphism of a smooth projective variety $X$ defined over an algebraically closed field $\k$. Then $h_t(\mathbb{L}f^*)$ is constant with value
\[h^\cat(\mathbb{L}f^*)=\max_{p} \log d_p(f)=\log \rho(f^*_{N^*}).\]
If $\k=\mathbb{C}$, they are also equal to $\log\rho(f^*_{H^*})=h_\top(f)$.\\
Moreover, if $\cD^b(X)$ admits a stability condition, then the previous quantities are equal to the mass growth rate $h_\sigma(\mathbb{L}f^*)$ for any numerical stability condition $\sigma$.
\end{thm}

\begin{rmk}[Numerical Chow and numerical Grothendieck]\label{rmk:NumericalGps}
In the above situation, observe that the Mukai-vector map induces an $f^*$-equivariant isomorphism between the numerical Grothendieck group $\mathcal{N}(X)_\mathbb{Q}\coloneqq\mathcal{N}(\cD^b(X))_\mathbb{Q}$ and the numerical Chow group $N^*(X)_\mathbb{Q}\coloneqq \operatorname{CH}^*(X)/\equiv$, both with rational coefficients, 
\[\operatorname{v}: \mathcal{N}(X)_\mathbb{Q}\xrightarrow{\cong} N^*(X)_\mathbb{Q}.\]
As a consequence, $\rho(\mathcal{N}(\mathbb{L}f^*))=\rho(f^*_{N^*})$ and  $s(\mathcal{N}(\mathbb{L}f^*))=s(f^*_{N^*})$.
\end{rmk}

The main result of this section is the following analogy of Theorem \ref{thm:KTO} for the polynomial entropy.
\begin{thm}\label{thm:EntropyOfMorphism}
Let $f$ be a surjective endomorphism of a smooth projective variety $X$ defined over an algebraically closed field $\k$. Let $\mathbb{L}f^*\colon \cD^b(X)\to \cD^b(X)$ be the derived pullback functor. Then the categorical polynomial entropy function of $\mathbb{L}f^*$ is constant with value
\begin{equation}
    \ph(\mathbb{L}f^*)=\max s_p(f)=s(f^*_{N^*}),
\end{equation}
where the maximum runs over all integers $p$ such that $d_p(f)$ attains the maximal dynamical degree of $f$. Here $s_p(f)$ denotes the $p$-th polynomial dynamical degree of $f$ (see Section \ref{sec:PolDynamicalDegree}), and $s(f^*_{N^*})$ is the polynomial growth rate of the induced action on the numerical Chow group $N^*(X)\coloneqq \bigoplus_p N^p(X)$.
\end{thm}
\begin{proof}
First of all, the fact that the function $\ph_t(\mathbb{L}f^*)$ is constant follows from Lemma \ref{lemm:ConstantFunction}, since $\mathbb{L}f^*=f^*$ preserves the standard t-structure on $\cD^b(X)$.
Let $d$ be the dimension of $X$. By Fujita's vanishing theorem (\cite{Fujita}, see also \cite[1.4.35]{LazarsfeldI}), there exists a very ample line bundle $L$ such that $$H^i(X, L\otimes L')=0,$$ for all $i>0$ and all nef line bundle $L'$. Take split generators $G=\bigoplus_{j=1}^{d+1} L^{-j}$ and $G'=G^\vee=\bigoplus_{j=1}^{d+1} L^{j}$ of $\cD^b(X)$ (\cite[Theorem 4]{Orlov}). By the choice of $L$, we have that 
\[\operatorname{Ext}^i(G, (f^n)^*(G'))=H^i(X, G^\vee\otimes (f^n)^*(G'))=0,\]
for all $n>0$ and all $i>0$.

Therefore, we obtain that
$$\epsilon(G, (\mathbb{L}f^*)^n(G'))=\dim H^0(X, G^{\vee}\otimes (f^*)^n(G'))=\chi(G, (f^*)^n(G')).$$
Since $\chi(-,-)$ is a non-degenerate bilinear form on $\mathcal{N}(X)$, there exists a constant $C>0$ such that 
\[\chi(G, (f^*)^n(G'))\leq C\cdot\|\mathcal{N}(\mathbb{L}f^*)^n\|.\]
It yields that 
\begin{align*}
\ph(\mathbb{L}f^*)
&=\limsup_{n\to \infty}  \frac{\log\epsilon(G, (\mathbb{L}f^*)^n(G'))-n h^\cat(\mathbb{L}f^*)}{\log(n)}\\
&=\limsup_{n\to \infty}  \frac{\log\epsilon(G, (\mathbb{L}f^*)^n(G'))-n \rho(\mathcal{N}(\mathbb{L}f^*))}{\log(n)}\\
&\leq\limsup_{n\to \infty}  \frac{\log C+ \log\|\mathcal{N}(\mathbb{L}f^*)^n\|-n \rho(\mathcal{N}(\mathbb{L}f^*))}{\log(n)}\\
&=s(\mathcal{N}(\mathbb{L}f^*)),
\end{align*}
where the first equality uses Lemma \ref{lemm:EntropyInExt-distance} (note that $\cD^b(X)$ is saturated), and the second equality uses Theorem \ref{thm:KTO} (and Remark \ref{rmk:NumericalGps}).

On the other hand, by Theorem \ref{thm:KTO}, the hypothesis in Proposition \ref{prop:GYLowerBoundPolEntropy} is satisfied. Hence we get the Yomdin-type lower bound:
\[\ph(\mathbb{L}f^*)\geq s(\mathcal{N}(\mathbb{L}f^*)).\]
We conclude that this is an equality; hence, $\ph(\mathbb{L}f^*)=s(f^*_{N^*})$ by Remark \ref{rmk:NumericalGps}. The remaining assertions follow from Proposition \ref{prop:PolDynDegCohomology}.
\end{proof}
Combining the previous theorem with our discussion on polynomial dynamical degrees in Section \ref{sec:PolDynamicalDegree}, we deduce the following consequence.
\begin{cor}[Complex setting]\label{cor:polycatautomorphism}
If $f$ is a surjective holomorphic endomorphism of a projective complex manifold $X$, then the categorical polynomial entropy of $\mathbb{L}f^*$ is equal to the polynomial growth rate of the induced action $f^*$ on the cohomology $H^*(X, \mathbb{Q})$:
\[\ph(\mathbb{L}f^*)=s(f^*).\]
In particular, if the topological entropy of $f$ is zero, then $\ph_t(\mathbb{L}f^*)$ is a constant function with value
$$\ph(\mathbb{L}f^*)=\lim_{n\to \infty} \frac{\log\|(f^*)^n\|}{\log(n)}=s(f^*),$$
where $s(f^*)+1$ is the maximal size of the Jordan blocks of $f^*$.
\end{cor}
\begin{proof}
The first part can be easily deduced by combining Proposition \ref{prop:PolDynDegCohomology} with Theorem \ref{thm:EntropyOfMorphism}. Now assuming the vanishing of the topological entropy, the Gromov--Yomdin Theorem \ref{thm:GromovYomdin} implies that all dynamical degrees are equal to 1 and the spectral radius of $f^*$ on $H^*(X)$ is 1. Thanks to the integral structure on cohomology, all the eigenvalues of $f^*$ are algebraic integers, hence they are roots of unity by Kronecker's theorem. Therefore, $f^*$ is virtually unipotent and  $\ph(\mathbb{L}f^*)=s(f^*)$
has the description by Jordan blocks in the statement, thanks to Lemma \ref{lemma:PolGrowLim}.
\end{proof}

\begin{rmk}
Corollary \ref{cor:polycatautomorphism} should be compared with the inequalities in \cite[Theorem 2.1 and Theorem 4.1]{CanPR}. Note that unlike its categorical counterpart, the topological polynomial entropy of an automorphism can indeed be different from the polynomial growth rate of the induced action on cohomology. For instance, the automorphism of the projective line given by $[x: y]\mapsto [x+y: y]$ has topological polynomial entropy 1, while its action on cohomology is trivial.
\end{rmk}

\begin{cor}[Polynomial mass growth]\label{cor:MassGrowthMorphism}
In the same situation as in Theorem \ref{thm:EntropyOfMorphism}, if $\cD^b(X)$ admits a numerical stability condition $\sigma$, then
$$\ph(\mathbb{L}f^*)=\ph_{\sigma}(\mathbb{L}f^*)=s(f^*_{N^*}).$$
\end{cor}
\begin{proof}
By the last assertion of Theorem \ref{thm:KTO}, one can apply Lemma \ref{lemma:ComparingEntropyMass} to see that 
\begin{equation}\label{eqn:InequalityMorphism}
    \ph(\mathbb{L}f^*)\geq \ph_\sigma(\mathbb{L}f^*)\geq s(\mathcal{N}(\mathbb{L}f^*)),
\end{equation}
where the second inequality is the Yomdin-type estimate in Proposition \ref{prop:GYLowerBoundMassGrowth}.
Therefore Theorem \ref{thm:EntropyOfMorphism} implies that both inequalities in \eqref{eqn:InequalityMorphism} are equalities. Finally, one can identify $s(f^*_{N^*})$ with $s(\mathcal{N}(\mathbb{L}f^*))$ by Remark \ref{rmk:NumericalGps}.
\end{proof}

\section{Examples}\label{sec: Examples}
  We compute in this section the categorical polynomial entropy, as well as the polynomial mass growth rate of some standard functors, in a parallel way to \cite[Section 2]{DHKK}.
Recall that when concentrating in the value of entropy functions at $t=0$, we write $\delta=\delta_0$, $\epsilon=\epsilon_0$, $m_{\sigma}=m_{\sigma, 0}$, $h^\cat=h_0$, $\ph=\ph_0$, $h_\sigma=h_{\sigma, 0}$, and $\ph_\sigma=\ph_{\sigma, 0}$, etc.

\subsection{Shifts}
The following lemma shows that cohomological shifts do not affect the polynomial entropy.
\begin{lemma}[Shifts (I):  entropy]\label{lemma:Shifts}
Let $\cD$ be a saturated triangulated category and $F$ an endofunctor of $\cD$. Then for any integer $m$, $$\ph_t(F\circ [m])=\ph_t(F).$$
In particular, $\ph_t([m])=0$.
\end{lemma}
\begin{proof}
In \cite[Lemma 3.7]{KiTa}, it is shown that the categorical entropy of the shift functor $[m]$ is $h_t([m])=mt$ and more generally, for any endofunctor $F$, we have $h_t(F\circ [m])=h_t(F)+mt$.\\
Now it follows from definition that for any split generator $G$ of $\cD$, 
$$\epsilon_t(G, F^n(G)[mn])=\epsilon_t(G, F^n(G))\cdot e^{mnt}.$$
Therefore, using Lemma \ref{lemm:EntropyInExt-distance}, we see that 
\begin{equation*}
\ph_t(F\circ [m])=\limsup_{n\to\infty}\frac{\log(\epsilon_t(G, F^n(G))\cdot e^{mnt})-n(h_t(F)+mt)}{\log(n)}=\ph_t(F).\qedhere
\end{equation*}
\end{proof}

Similarly for the mass growth, we have the following result.
\begin{lemma}[Shifts (II): mass growth]\label{lemma:ShiftsMassGrowth}
Let $\cD$ be a triangulated category endowed with a stability condition $\sigma$. Let $F$ be a endofunctor of $\cD$. Then for any integer $m$, $$h_{\sigma, t}(F\circ [m])=h_{\sigma, t}(F)+mt; \quad \ph_{\sigma, t}(F\circ [m])=\ph_{\sigma, t}(F).$$
In particular, $h_{\sigma, t}([m])=mt$ and $\ph_{\sigma, t}([m])=0$.
\end{lemma}
\begin{proof}
The proof is the same as in Lemma \ref{lemma:Shifts}, by using the fact that $m_{\sigma, t}(E[m])=m_{\sigma, t}(E)e^{mt}$ for any object $E$ and any integer $m$.
\end{proof}

\begin{rmk}[Serre functor of fractional CY categories]
A triangulated category $\cD$ with a Serre functor $\mathbf{S}$ is called \textit{fractional Calabi--Yau}, if there exist integers $n> 0$ and $m$ such that $\mathbf{S}^n\cong [m]$. The rational number $\frac{m}{n}$ is called the Calabi--Yau dimension of $\cD$. In \cite[Section 2.6.1]{DHKK}, it is shown that if $\cD$ is saturated, then $$h_t(\mathbf{S})=\frac{m}{n}t.$$ As for the polynomial entropy, we claim that 
\[\ph_t(\mathbf{S})=0.\]
Indeed, fixing a split generator $G$ of $\cD$, for any $t\in \mathbb{R}$ and any $N>0$, if one writes $N=nq+r$ with $q=\lfloor\frac{N}{n}\rfloor$ and $0\leq r<n$, then we have 
\[\epsilon_t(G, \mathbf{S}^N(G))=\epsilon_t(G,\mathbf{S}^r(G)[mq])=\epsilon_t(G,\mathbf{S}^r(G))\cdot e^{mqt}.\]
Therefore
\[\log\epsilon_t(G, \mathbf{S}^N(G))-N\cdot h_t(\mathbf{S})=\log\epsilon_t(G,\mathbf{S}^r(G))-\frac{mr}{n}t.\]
To conclude, it suffices to observe that the absolute value of the right-hand side is bounded, independently of $N$, by 
\[\max_{0\leq r\leq n-1}\{|\log\epsilon_t(G,\mathbf{S}^r(G))|\}+|mt|.\]
Similarly, for the mass growth rate, instead of assuming the saturatedness, we suppose there exists a stability condition $\sigma$ on $\cD$, then $$h_{\sigma, t}(\mathbf{S})=\frac{m}{n}t \quad \text{ and }\quad  \ph_{\sigma, t}(\mathbf{S})=0.$$
As a consequence, for any endofunctor $F$ commuting with $\mathbf{S}$ (for example, an autoequivalence), Lemma \ref{lemma:composition} implies that 
\[h_t(F\circ \mathbf{S})=h_t(F)+\frac{m}{n}t  \quad \text{and} \quad \ph_t(F\circ \mathbf{S})=\ph_t(F).\]
\end{rmk}

\subsection{Tensoring line bundles}\label{sec:LineBundle}
Let the base field $\k$ be algebraically closed. Let $X$ be a smooth projective variety defined over $\k$ and $L$ a line bundle on $X$. We consider here the autoequivalence of tensoring with $L$ on the derived category of $X$:
\[-\otimes L\colon \cD^b(X)\to \cD^b(X).\]
It is shown in \cite{DHKK} that the categorical entropy function of this functor is the constant zero function:
\begin{equation}\label{eqn:CatEntropyTensorLineBundle}
    h_t(-\otimes L)=h^\cat(-\otimes L)=0.
\end{equation}
We shall study its categorical polynomial entropy, which is indeed a non-trivial and meaningful invariant. The first result below estimates the polynomial entropy.  Recall that the \textit{numerical dimension} of $L$ is by definition 
\[\nu(L)\coloneqq \max\{m \mid c_1(L)^m \not\equiv 0\},\]
where $\equiv$ denotes the numerical equivalence relation.
\begin{prop}\label{prop:TensorLineBundle}
Notation is as before. Let $d=\dim(X)$.
\begin{enumerate}[label=(\roman*)]
    \item The polynomial entropy function $\ph_t(-\otimes L)$ is constant in $t$.
    \item $\ph(-\otimes L)\leq d$.
    \item $\ph(-\otimes L)\geq \nu(L)$.
\end{enumerate}
\end{prop}
\begin{proof}
$(i)$ follows from Lemma \ref{lemm:ConstantFunction}, since $-\otimes L$ preserves the standard t-structure of $\cD^b(X)$.\\
$(ii)$. By definition and \eqref{eqn:CatEntropyTensorLineBundle}, the value of this constant function is
$$\ph(-\otimes L)=\limsup_{n\to \infty}\frac{\log \epsilon(G, G\otimes L^{\otimes n})}{\log(n)},$$
where $G$ is any split generator. In the sequel, we choose  $G$ to be a locally free sheaf on $X$, for example $\bigoplus_{i=0}^{d} \mathcal{O}_X(i)$. By definition,
\begin{equation}\label{eqn:SumHk}
    \epsilon(G, G\otimes L^{\otimes n})=\sum_{k=0}^{d}\dim H^k(X, G^\vee\otimes G\otimes L^{\otimes n}).
\end{equation}
By \cite[Example 1.2.33]{LazarsfeldI}, for any $k$, the growth of $\dim H^k(X, G^\vee\otimes G\otimes L^{\otimes n})$ is at most like a polynomial of degree $d$; hence $\epsilon(G, G\otimes L^{\otimes n})= O(n^d)$, and $(ii)$ follows immediately.\\
$(iii)$. We claim that the polynomial growth of the endomorphism 
\[\cdot [L]\colon \mathcal{N}(X)\to \mathcal{N}(X)\]
is the numerical dimension of $L$, where $\mathcal{N}(X)\coloneqq \mathcal{N}(\cD^b(X))$ is the numerical Grothendieck group of $X$.  Indeed, denoting by $N^*(X)_\mathbb{Q}\coloneqq \operatorname{CH}^*(X)_\mathbb{Q}/{\equiv}$ the Chow group of algebraic cycles modulo numerical equivalence, then
the Mukai-vector map 
\[\operatorname{v}\colon \mathcal{N}(X)_\mathbb{Q}\to N^*(X)_\mathbb{Q}\]
is an isomorphism of finite-dimensional $\mathbb{Q}$-vector spaces, making the following diagram commutative:
\begin{equation*}
    \xymatrix{
    \mathcal{N}(X)_\mathbb{Q} \ar[r]^{\cdot [L]} \ar[d]_{\operatorname{v}}^{\cong} & \mathcal{N}(X)_\mathbb{Q} \ar[d]_{\operatorname{v}}^{\cong}\\
    N^*(X)_\mathbb{Q} \ar[r]^{\cdot e^{c_1(L)}} & N^*(X)_\mathbb{Q},
    }
\end{equation*}
where the bottom arrow is the intersection product with 
$$e^{c_1(L)}=\sum_{k=0}^{\nu(L)}\frac{1}{k!}c_1(L)^k.$$
It is then clear that the spectral radius of $\cdot e^{c_1(L)}\in \operatorname{End}(N^*(X))$ is 1 and its polynomial growth rate is $\nu(L)$. The claim is proved. Combined with the Yomdin-type lower bound established in Proposition \ref{prop:GYLowerBoundPolEntropy} (note that $h^\cat(-\otimes L)=0$), we obtain that $\ph(-\otimes L)\geq s(\cdot [L])=\nu(L)$.
\end{proof}

Now we try to determine precisely the value of the polynomial entropy of the functor $- \otimes L$. It turns out to be quite related to the positivity properties of $L$. For illustration, let us first show the following result on big and nef line bundles. 

\begin{lemma}[Nef and big line bundles]\label{lemma:NefBig}
Let $L$ be a nef line bundle on a smooth projective variety $X$. If $L$ is moreover big, then 
$\ph(-\otimes L)= d$.\\
If  $L$ is not big, then $\ph(-\otimes L)\leq d-1$.
\end{lemma}
\begin{proof}This is a special case of Theorem \ref{thm:TensorNefLineBundle} below, but we prefer to give a direct proof here.
If $L$ is big and nef, then by \cite[Corollary 1.4.41]{LazarsfeldI}, $\dim H^0(X, G^\vee\otimes G\otimes L^{\otimes n})$ grows exactly as a polynomial of degree $d$. Hence $\epsilon(G, G\otimes L^{\otimes n})$ grows at least as a polynomial of degree $d$. Therefore $\ph(-\otimes L)\geq d$. We can conclude as Proposition \ref{prop:TensorLineBundle} $(ii)$ provides the other inequality.\\
If $L$ is nef but not big, on one hand, \cite[Theorem 1.4.40]{LazarsfeldI} implies that the growth of the $k$-th term of \eqref{eqn:SumHk} is at most like a polynomial of degree $d-k$:
$$\dim H^k(X, G^\vee\otimes G\otimes L^{\otimes n})= O(n^{d-k}).$$
On the other hand, the hypothesis implies that $(c_1(L)^d)=0$, then \cite[Corollary 1.4.41]{LazarsfeldI} says that the $0$-th term of \eqref{eqn:SumHk} also grows at most as a polynomial of degree $d-1$. As a result, $\epsilon(G, G\otimes L^{\otimes n})$ grows at most as a polynomial of degree $d-1$. Therefore $\ph(-\otimes L)\leq d-1$.
\end{proof}

To generalize Lemma \ref{lemma:NefBig}, we first establish the following result on the growth of cohomology of powers of nef line bundles, which will be used in the proof of Theorem \ref{thm:TensorNefLineBundle}. It is probably well-known to experts, but as we cannot find a reference, a proof is provided for the sake of completeness. 
\begin{prop}\label{prop:GrowthNefDivisor}
Let $X$ be a smooth projective variety over an algebraically closed field $\k$. Let $L$ be a nef line bundle and $\mathcal{F}$ be a coherent sheaf on $X$. Then
\[\dim H^k(X, \mathcal{F}\otimes L^{\otimes n})= O(n^{\nu(L)}),\]
that is, there exists a constant $C>0$ such that for all $k\geq 0$ and $n\geq 0$, we have
\[\dim H^k(X, \mathcal{F}\otimes L^{\otimes n})\leq  Cn^{\nu(L)},\]
where $\nu(L)$ is the numerical dimension of $L$.
\end{prop}
\begin{proof}
The proof is similar to \cite[Theorem 1.4.40]{LazarsfeldI}. We proceed by induction on the dimension of $X$. By Fujita's vanishing theorem \cite{Fujita}, there is a very ample divisor $H$ such that 
\begin{equation}\label{eqn:FujitaVanishing}
    H^k(X, \mathcal{F}\otimes \mathcal{O}_X(H)\otimes L^{\otimes n})=0
\end{equation}
for all $k>0$ and all $n\geq 0$.
Up to replacing $H$ by a general member in its linear system, we can assume that $H$ is smooth and does not contain any subvariety of $X$ defined by associated primes of $\mathcal{F}$. Therefore we have the exact sequence
\[0\to \mathcal{F}\otimes L^{\otimes n}\to \mathcal{F}\otimes L^{\otimes n}\otimes \mathcal{O}_X(H)\to\mathcal{F}\otimes L^{\otimes n}\otimes \mathcal{O}_X(H)\otimes O_H\to 0 \]
for any $n\geq 0$. Let us look at the corresponding long exact sequence of cohomology groups.

For any $k>0$, by the vanishing hypothesis \eqref{eqn:FujitaVanishing}, we have $$\dim H^k(X, \mathcal{F}\otimes L^{\otimes n})\leq \dim H^{k-1}(H,\mathcal{F}\otimes L^{\otimes n}\otimes \mathcal{O}_X(H)\otimes O_H),$$
which is $O(n^{\nu(L|_H)})$ by the induction hypothesis. As $\nu(L|_H)\leq \nu(L)$, we conclude that $\dim H^k(X, \mathcal{F}\otimes L^{\otimes n})=O(n^{\nu(L)})$.

For $k=0$, again by \eqref{eqn:FujitaVanishing},
$$\dim H^0(X, \mathcal{F}\otimes L^{\otimes n})\leq \dim H^{0}(X,\mathcal{F}\otimes L^{\otimes n}\otimes \mathcal{O}_X(H))=\chi(X, \mathcal{F}\otimes L^{\otimes n}\otimes \mathcal{O}_X(H)).$$
Thanks to the Hirzebruch--Riemann--Roch formula, $\chi(X, \mathcal{F}\otimes L^{\otimes n}\otimes \mathcal{O}_X(H))$ is a polynomial in $n$ of the form $$\sum_{i=0}^d n^i\int_X c_1(L)^i\alpha_i,$$
where $\alpha_i$ is some $i$-dimensional algebraic cycle of $X$. As $c_1(L)^i\equiv 0$ for all $i>\nu(L)$, the above polynomial is of degree at most $\nu(L)$. Therefore we can conclude that $\dim H^0(X, \mathcal{F}\otimes L^{\otimes n})=O(n^{\nu(L)})$.

The induction process is complete.
\end{proof}

We are ready to determine the polynomial entropy for tensoring a nef line bundle.
\begin{thm}[Nef line bundles]\label{thm:TensorNefLineBundle}
Let $L$ be a line bundle on a smooth projective variety $X$ defined over an algebraically closed field $\k$. If $L$ or $L^{-1}$ is nef, then the polynomial entropy of the functor $-\otimes L$ is equal to the numerical dimension of $L$: $$\ph(-\otimes L)=\nu(L).$$
\end{thm}
\begin{proof}
We only need to treat the case where $L$ is nef, as the anti-nef case follows from the nef case by using Lemma \ref{lemma:Inverse}. Keep the notation in the proof of Proposition \ref{prop:TensorLineBundle}. Applying Proposition \ref{prop:GrowthNefDivisor} to the case $\mathcal{F}=G^\vee\otimes G$, we obtain that each term on the right-hand side of \eqref{eqn:SumHk} is $O(n^{\nu(L)})$, thus so is $\epsilon(G, G\otimes L^{\otimes n})$. Consequently, $\ph(-\otimes L)\leq \nu(L)$.
Combined with the lower bound in Proposition \ref{prop:TensorLineBundle} $(iii)$,  we must have an equality.
\end{proof}

The nefness assumption in Theorem  \ref{thm:TensorNefLineBundle} can not be removed. Let us give a simple example of line bundle whose polynomial entropy is strictly bigger than its numerical dimension. 
\begin{eg}\label{ex:GYCouldBeStrict}
Let $S$ be a smooth projective surface and $H$ be an ample divisor on it. Assume that the degree of $S$ is a square, that is, $(H^2)=m^2$ for some $m\in \mathbb{Z}_{>0}$; for example when $S=\mathbb{P}^2$ and $H=c_1(\mathcal{O}_{\mathbb{P}^2}(1))$. Let $\tau\colon S'\to S$ be the blow up of $S$ at a point. Let $H'\coloneqq \tau^*(H)$ be the pullback of the polarization and $E$ the exceptional divisor. We have that $({H'}^2)=m^2$, $(H'\cdot E)=0$, and $(E^2)=-1$. Consider $D\coloneqq H'+mE$. Hence $(D^2)=0$ and the numerical dimension of $\mathcal{O}_{S'}(D)$ is 1. On the other hand, we claim that $\ph(-\otimes \mathcal{O}_{S'}(D))=2$. Indeed, let us fix split generators $G=\bigoplus_{i=0}^{2}\mathcal{O}(i)$ and $G'=G^\vee$ where $\mathcal{O}(1)$ is a very ample line bundle on $S'$ (see \cite{Orlov}). Then $\epsilon(G', G\otimes \mathcal{O}_{S'}(nD))\geq h^0(S', G\otimes G\otimes \mathcal{O}_{S'}(nD))\geq h^0(S', \mathcal{O}_{S'}(nD)^{\oplus 9})$, which has quadratic growth rate in $n$, since $D$ is a big divisor. 
Hence $\ph(-\otimes \mathcal{O}_{S'}(D))=2$. By taking product of $S'$ with another variety and pulling back the divisor to the product, one can also produce examples where the polynomial entropy is strictly between the numerical dimension and the dimension.
\end{eg}

\begin{cor}[Numerical nature of entropy]\label{cor:NumericalNature}
Let $L$ be a line bundle on a smooth projective variety $X$. The polynomial entropy $\ph(-\otimes L)$ depends only on the numerical class  of $L$.
\end{cor}
\begin{proof}
For any numerically trivial (in particular nef) line bundle $L_0$, Theorem \ref{thm:TensorNefLineBundle} implies that $h^\cat(-\otimes L_0)=\ph(-\otimes L_0)=0$. Applying Lemma \ref{lemma:composition} to the commuting functors $F_1=-\otimes L_0$ and $F_2=-\otimes L$, we see that $\ph(-\otimes L\otimes L_0)=\ph(-\otimes L)$.
\end{proof}

We can deduce from the above results some information about the mass growth of the functor $-\otimes L$:
\begin{cor}
Let $L$ be a line bundle on a smooth projective variety $X$. Assume that $\cD^b(X)$ admits a stability condition $\sigma$ that factors through the numerical Grothendieck group. Then $h_\sigma (-\otimes L)=0$ and $$\ph(-\otimes L)\geq \ph_\sigma(-\otimes L)\geq \nu(L).$$ In particular, if $L$ or $L^{-1}$ is nef, then $\ph_\sigma(-\otimes L)=\nu(L)$.
\end{cor}
\begin{proof}
By \cite[Theorem 3.5(2) and Proposition 3.11]{Ikeda} and the fact that $h^\cat(-\otimes L)=0$, we have $h_\sigma (-\otimes L)=0$. Hence the hypothesis of Lemma \ref{lemma:ComparingEntropyMass} is satisfied for $t=0$, and we have $\ph(-\otimes L)\geq \ph_\sigma(-\otimes L).$ The lower bound $\ph_\sigma(-\otimes L)\geq \nu(L)$ is actually the Yomdin-type inequality established Proposition \ref{prop:GYLowerBoundMassGrowth}, since $\nu(L)$ is the polynomial growth rate of the endomorphism $\cdot [L]$ of $\mathcal{N}(X)$ by the proof of Proposition \ref{prop:TensorLineBundle} $(iii)$. The last statement follows from Theorem \ref{thm:TensorNefLineBundle}.
\end{proof}


\begin{rmk}[Serre functor for a variety]
Let $X$  be a smooth projective variety. The Serre functor of the derived category $\cD^b(X)$ is given by $\mathbf{S}_X=-\otimes \omega_X[\dim(X)]$. Its categorical entropy is computed in \cite[\S 2.6.2]{DHKK}, namely, $h_t(\mathbf{S}_X)=\dim(X)t$.
As for its polynomial entropy, we see a close link with the birational geometry of $X$:  using Lemma \ref{lemma:Shifts}, Proposition \ref{prop:TensorLineBundle}, and Theorem \ref{thm:TensorNefLineBundle}, we obtain that $\ph_t(\mathbf{S}_X)$ is a constant function with value 
\[\ph(\mathbf{S}_X)=\ph(-\otimes \omega_X),\] which is 
\begin{itemize}
    \item between the numerical dimension of $\omega_X$ and $\dim(X)$;
    \item equal to $\dim(X)$ if $X$ is minimal and of general type;
    \item equal to $\dim(X)$ if 
$\omega_X^\vee$ is nef and big, i.e.~$X$ is a weak Fano variety;
    \item equal to the numerical dimension of $\omega_X$ if $X$ is a minimal model. Assuming the abundance conjecture when $\operatorname{char}(\k)=0$, then it is equal to the Kodaira dimension of $X$.
\end{itemize}

\end{rmk}

\subsection{Spherical twists (I)}
\label{sec:sphericaltwists}
Let $\cD$ be a saturated triangulated category endowed with a Serre functor $\mathbf{S}$. Recall that an object $\mathcal{E}$ in $\cD$ is called $d$-\textit{spherical} for some positive integer $d$, if $\mathbf{S}(\mathcal{E})\cong\mathcal{E}[d]$ and $\dim\Hom_\cD(\mathcal{E}, \mathcal{E}[*])=\dim H^*(\mathbb{S}^d, \mathbb{Q})$, where $\mathbb{S}^d$ is the $d$-dimensional topological sphere. Examples of spherical objects are line bundles in the derived category of Calabi--Yau varieties, the structure sheaf of (-2)-curves on K3 surfaces etc. 

The most interesting feature of spherical objects is that they induce  autoequivalences of the triangulated category \cite{SeidelThomas, MR3692883},  called \textit{spherical twists}. More precisely, for a $d$-spherical object $\mathcal{E}$, the spherical twist around $\mathcal{E}$ is the autoequivalence 
\begin{eqnarray*}
T_\mathcal{E}\colon \cD&\to& \cD\\
E &\mapsto& \operatorname{Cone}(\operatorname{RHom}(\mathcal{E}, E)\otimes \mathcal{E}\to E).
\end{eqnarray*}
It follows from definition that $T_\mathcal{E}(\mathcal{E})\cong \mathcal{E}[1-d]$ and $T_\mathcal{E}|_{\mathcal{E}^\perp}\cong \operatorname{id}.$

To study the polynomial entropy of spherical twists, we need the following estimate, which improves upon \cite[Proof of Theorem 3.1]{Ouchi}.

\begin{lemma}[Upper bound]\label{lemma:UpperBound}
For any objects $G, G'\in \cD$ and any positive integer $n$, we have
\[\epsilon_t(G', T^n_\mathcal{E}(G))\leq 
\begin{cases}
ne^tA_t+B_t &\text{ if } d=1;\\
nA_t+B_t &\text{ if } t=0;\\
\frac{e^{(1-d)nt}}{e^{(1-d)t}-1}A_t+B_t &\text{ if } t<0 \text{ and } d\geq 2;\\
\frac{e^t}{1-e^{(1-d)t}} A_t+B_t &\text{ if } t>0 \text{ and } d\geq 2;\\
\end{cases}\]
where $A_t=\epsilon_t(G', \operatorname{RHom}(\mathcal{E}, G)\otimes \mathcal{E})$ and $B_t=\epsilon_t(G', G)$ are positively valued functions, which are independent  of $n$.
\end{lemma}
\begin{proof}
For any $n>0$, applying $T_\mathcal{E}^{n-1}$ to the distinguished triangle
\[\operatorname{RHom}(\mathcal{E}, G)\otimes \mathcal{E}\to G\to T_\mathcal{E}(G)\xrightarrow{+1},\]
we get 
\[\operatorname{RHom}(\mathcal{E}, G)\otimes \mathcal{E}[(1-d)(n-1)]\to T^{n-1}_\mathcal{E}(G)\to T^n_\mathcal{E}(G)\xrightarrow{+1}.\]
Therefore,
$\epsilon_t(G', T^n_\mathcal{E}(G))\leq \epsilon_t(G',  T^{n-1}_\mathcal{E}(G))+A_te^{((1-d)n+d)t}$,
hence
\[\epsilon_t(G', T^n_\mathcal{E}(G))\leq \epsilon_t(G', G)+A_t\sum_{i=1}^ne^{((1-d)i+d)t}=B_t+A_t\sum_{i=1}^ne^{((1-d)i+d)t}.\]
The cases when $d=1$ or $t=0$ follow immediately. Assume now $d\geq 2$ and $t\neq 0$. Then 
\[\epsilon_t(G', T^n_\mathcal{E}(G))\leq B_t+A_t\frac{e^t}{e^{(1-d)t}-1}(e^{(1-d)tn}-1).\]
One can conclude easily by separating  the cases $t>0$ and $t<0$.
\end{proof}

\begin{prop}\label{prop:SphercialTwist}
Let $d$ be a positive integer.  Let $\mathcal{E}$ be a $d$-spherical object in a saturated triangulated category $\cD$.
\begin{enumerate}[label=(\roman*)]
    \item If $d=1$, then $0\leq \ph_t(T_\mathcal{E})\leq 1$.
    \item If $t=0$, then $0\leq \ph(T_\mathcal{E})\leq 1$.
    \item  If $d\geq 2$ and $t< 0$, then $\ph_t(T_\mathcal{E})=0$.
    \item If $d\geq 2$ and $t>0$, assuming moreover that $\mathcal{E}^\perp\coloneqq
\{E\in \cD\mid \Hom_\cD(\mathcal{E}, E[k])=0 \text{ for all } k\in \mathbb{Z}\}\neq 0$, then $\ph_t(T_\mathcal{E})=0$.
\end{enumerate}
\end{prop}
\begin{proof}
In any of the cases $(i) \sim (iv)$, Ouchi \cite{Ouchi} computed the categorical entropy of the spherical twist $T_\mathcal{E}$:
\[h_t(T_\mathcal{E})=
\begin{cases} (1-d)t, &\quad \text{if } t\leq 0;\\ 0 & \quad\text{if } t>0.
\end{cases}\]
Fix two split generators $G, G'$ of $\cD$. We first establish, in any case of $(i) \sim (iv)$, the lower bound that
\[\ph_t(T_\mathcal{E})\geq 0\]
For $(i)\sim (iii)$, $h_t(T_\mathcal{E})=(1-d)t$. As $G\oplus \mathcal{E}$ is also a split generator,
\begin{align*}
\ph_t(T_\mathcal{E})&=\limsup_{n \to \infty}\frac{\log \epsilon_t(G', T^n_\mathcal{E}(G\oplus \mathcal{E}))-h_t(T_\mathcal{E})}{\log(n)}\\
&\geq \limsup_{n \to \infty}\frac{\log \epsilon_t(G', T^n_\mathcal{E}(\mathcal{E}))-n(1-d)t}{\log(n)}\\
&=\limsup_{n \to \infty}\frac{\log \epsilon_t(G', \mathcal{E}[(1-d)n])-n(1-d)t}{\log(n)}\\
&=\limsup_{n \to \infty}\frac{\log \epsilon_t(G', \mathcal{E})}{\log(n)}\\
&=0.
\end{align*}
Here we used the fact that $T_\mathcal{E}(\mathcal{E})\cong \mathcal{E}[1-d]$.\\
As for $(iv)$, take  $E$ a non-zero object in $\mathcal{E}^\perp$, as $G\oplus E$ is also a split generator,
\begin{align*}
\ph_t(T_\mathcal{E})&=\limsup_{n \to \infty}\frac{\log \epsilon_t(G', T^n_\mathcal{E}(G\oplus E))-h_t(T_\mathcal{E})}{\log(n)}\\
&\geq \limsup_{n \to \infty}\frac{\log \epsilon_t(G', T^n_\mathcal{E}(E))}{\log(n)}\\
&=\limsup_{n \to \infty}\frac{\log \epsilon_t(G', E)}{\log(n)}\\
&=0,
\end{align*}
where we used the fact that $T_\mathcal{E}(E)\cong E$.\\

Let us now establish the upper bounds in the statement:\\
$(i)$ When $d=1$, $h_t(T_\mathcal{E})=0$ and Lemma \ref{lemma:UpperBound} says that $$\epsilon_t(G', T^n_\mathcal{E}(G))\leq ne^tA_t+B_t.$$
Therefore,
\[\ph_t(T_\mathcal{E})=\limsup_{n \to \infty}\frac{\log \epsilon_t(G', T^n_\mathcal{E}(G))}{\log(n)}\leq \frac{\log(ne^tA_t+B_t)}{\log(n)}=1.\]
$(ii)$. When $t=0$, Lemma \ref{lemma:UpperBound} implies that 
\[\ph(T_\mathcal{E})=\limsup_{n \to \infty}\frac{\log \epsilon_t(G', T^n_\mathcal{E}(G))}{\log(n)}\leq \frac{\log(nA_t+B_t)}{\log(n)}=1.\]
$(iii)$. When $t<0$ and $d\geq 2$, using Lemma \ref{lemma:UpperBound},  we see that 
\begin{align*}
\ph_t(T_\mathcal{E})&=\limsup_{n \to \infty}\frac{\log \epsilon_t(G', T^n_\mathcal{E}(G))-nh_t(T_\mathcal{E})}{\log(n)}\\
&\leq\limsup_{n\to \infty} \frac{\log(\frac{e^{(1-d)nt}}{e^{(1-d)t}-1}A_t+B_t)-n(1-d)t}{\log(n)}\\
&=0.
\end{align*}
$(iv)$ When $t>0$ and $d\geq 2$, assuming $\mathcal{E}^\perp\neq 0$, then Lemma \ref{lemma:UpperBound} yields that
\begin{align*}
\ph_t(T_\mathcal{E})&=\limsup_{n \to \infty}\frac{\log \epsilon_t(G', T^n_\mathcal{E}(G))-nh_t(T_\mathcal{E})}{\log(n)}\\
&\leq\limsup_{n\to \infty} \frac{\log(\frac{e^t}{1-e^{(1-d)t}} A_t+B_t)}{\log(n)}\\
&=0.\qedhere
\end{align*}
\end{proof}
\begin{rmk}
When $\cD$ is the derived category of a projective K3 surface, the condition on the orthogonal complement in Proposition \ref{prop:SphercialTwist} $(iv)$ is often satisfied (for example when Picard number is 1), see Bayer's appendix of \cite{Ouchi}.
\end{rmk}

\subsection{Spherical twists (II): quiver Calabi--Yau categories}
\label{sec:sphericaltwistquiver}
Let $Q$ be an acyclic quiver with vertices labelled by $\{1,2,\ldots,n\}$, and let $\cD_Q$ be the $3$-Calabi--Yau category constructed from the Ginzburg Calabi--Yau dg-algebra associated to the quiver $Q$ \cite{Ginzburg,Keller}.
For each vertex $1\leq i\leq n$, there is an associated spherical object $S_i\in\cD_Q$.
We denote $T_i$ the spherical twist associated to $S_i$.
The morphisms between spherical objects are determined by the quiver: if there are $e_{ij}$ arrows from vertex $i$ to vertex $j$, then $\Hom^\bullet(S_i,S_j) = \bC^{\oplus e_{ij}}[-1]$.
By Lemma \ref{lemma:algstabilitycond} and the same computations as in \cite[Section~4.2]{Ikeda}, we have
\[
h^\pol(T_i^k) = \lim_{n\ra\infty} \frac{\log\ell(n)}{\log n} = 1
\]
for any power $k$ of the spherical twist $T_i$.
Here $\ell(n)$ is a linear polynomial which depends on $k$ and the valency of the $i$-th vertex.

We focus on the case of the $3$-Calabi--Yau category $\cD=\cD_{A_2}$ associated to the $A_2$-quiver
\[
\bullet \ra \bullet.
\]
The subgroup of $\Aut(\cD)$ generated by the spherical twists $T_1,T_2$ is isomorphic to the standard braid group on $3$ strings
\[
\left< T_1, T_2 \right> \cong \mathrm{Br}_3 =
\left< \sigma_1, \sigma_2: \sigma_1\sigma_2\sigma_1=\sigma_2\sigma_1\sigma_2 \right>,
\]
cf.~\cite[Section~2]{BQS20}.
There is a short exact sequence
\[
1\ra\bZ\ra\left< T_1, T_2 \right>\ra\mathrm{PSL}(2,\bZ)\ra1,
\]
where the map $\bZ\ra\left< T_1, T_2 \right>$ is given by sending $1$ to $(T_1T_2)^3=[5]$, and the map
\[
\phi\colon\left< T_1, T_2 \right>\ra\mathrm{PSL}(2,\bZ)
\]
is given by the induced action on the numerical Grothendieck group $\cN(\cD)$:
\[
T_1\mapsto
\begin{pmatrix}
1 & 1\\
0 & 1
\end{pmatrix}
\text{, }
T_2\mapsto
\begin{pmatrix}
1 & 0\\
-1 & 1
\end{pmatrix}
\]
with respect to the basis $\{[S_1],[S_2]\}$ of $\cN(\cD)$.

We prove the following trichotomy.

\begin{prop}
Let $F\in\left<T_1,T_2\right>\subset\Aut(\cD)$.
Then
    \begin{enumerate}[label=(\roman*)]
        \item $h^\pol(F)=h^\cat(F)=0$ if and only if
        $\cN(F)$ is elliptic (i.e.~$|\tr(\cN(F))|<2$) or 
        $\cN(F)=\pm\mathrm{id}$. In this case, $F^n=[m]$ for some integers $n,m$.
        
        \item $h^\pol(F)>0$ and $h^\cat(F)=0$ if and only if
        $\cN(F)$ is parabolic (i.e.~$|\tr(\cN(F))|=2$)
        and $\cN(F)\neq\pm\mathrm{id}$. In this case, $\ph(F)=1$.
        
        \item $h^\cat(F)>0$ if and only if 
        $\cN(F)$ is hyperbolic (i.e.~$|\tr(\cN(F))|>2$).
        In this case, $F$ is \emph{pseudo-Anosov} in the sense of \cite{FFHKL}.
    \end{enumerate}
Moreover, the Gromov--Yomdin-type equality holds for both categorical entropy and categorical polynomial entropy of any $F\in\left<T_1,T_2\right>$:
\[
h^\cat(F) = \log \rho(\cN(F)) \text{ \ and \ } h^\pol(F) = s(\cN(F)).
\]
\end{prop}

\begin{proof}
$(i)$. An elliptic element in $\mathrm{PSL}(2,\bZ)$ is conjugate to either
\[
\begin{pmatrix}
0 & 1 \\
-1 & 0
\end{pmatrix}
\text{, }
\begin{pmatrix}
1 & 1 \\
-1 & 0
\end{pmatrix}
\text{, or }
\begin{pmatrix}
0 & 1 \\
-1 & -1
\end{pmatrix}.
\]
Observe that
\[
\phi(T_1T_2T_1)=
\begin{pmatrix}
0 & 1 \\
-1 & 0
\end{pmatrix},
\ 
\phi(T_2T_1)=
\begin{pmatrix}
1 & 1 \\
-1 & 0
\end{pmatrix},
\text{ and }
\phi((T_2T_1)^2)=
\begin{pmatrix}
0 & 1 \\
-1 & -1
\end{pmatrix}.
\]
Hence if $\cN(F)$ is elliptic, then
\[
F = g F' g^{-1} [k]
\]
for some $g\in\left<T_1,T_2\right>$, $F'\in\{T_1T_2T_1,\ T_2T_1,\ (T_2T_1)^2\}$, and $k\in\bZ$.
Since
\[
T_1T_2T_1,\ T_2T_1,\ (T_2T_1)^2
\]
are of finite order up to shifts, hence $F^n=[m]$ for some integers $n,m$.
In this case, we have
\[
h^\cat(F)=h^\pol(F)=\log\rho(\cN(F))=s(\cN(F))=0.
\]
$(ii)$. A parabolic element in $\mathrm{PSL}(2,\bZ)$ is conjugate to
\[
\phi(T_1^n) = 
\begin{pmatrix}
1 & n \\
0 & 1
\end{pmatrix}
\]
for some $n\in\bZ$.
Hence if $\cN(F)$ is parabolic and $\cN(F)\neq\pm\mathrm{id}$, then
\[
F = g T_1^n g^{-1} [k]
\]
for some $g\in\left<T_1,T_2\right>$, $n\in\bZ\bs\{0\}$, and $k\in\bZ$.
By the previous computations, we have
\[
h^\cat(F)=0 \text{ \ and \ } h^\pol(F)=1.
\]
Since $\cN(F)$ in this case is quasi-unipotent and has a single Jordan block of size $2$, we have
\[
h^\cat(F)=\log \rho(\cN(F))=0 \text{ \ and \ } h^\pol(F)=s(\cN(F))=1.
\]
$(iii)$. The spectral radius of a hyperbolic element is greater than $1$.
Thus, if $\cN(F)$ is hyperbolic, then
$h^\cat(F) \geq \log\rho(\cN(F)) > 0$.
We will show that in fact the equality holds.
Any conjugacy class in $\mathrm{PSL}(2,\bZ)$ of infinite order has a representative of the form
\[
\begin{pmatrix}
1 & 1\\
0 & 1
\end{pmatrix}^{a_1}
\begin{pmatrix}
1 & 0\\
1 & 1
\end{pmatrix}^{b_1}
\cdots
\begin{pmatrix}
1 & 1\\
0 & 1
\end{pmatrix}^{a_n}
\begin{pmatrix}
1 & 0\\
1 & 1
\end{pmatrix}^{b_n}
\]
for some non-negative integer exponents $a_1,b_1,\ldots,a_n,b_n\in\bZ$ (see for instance \cite[Proposition~2.3]{Dunbar}).
Therefore, if $\cN(F)$ is hyperbolic, then $F$ is conjugate to an autoequivalence of the form
\[
T_1^{a_1}T_2^{-b_1}\cdots T_1^{a_n}T_2^{-b_n}
\]
up to shifts.
By the computations in \cite[Theorem~3.1]{FFHKL},
the mass growth and the polynomial mass growth of the split generator $S_1\oplus S_2$ of such an autoequivalence
coincide with the spectral radius and the polynomial growth rate of the corresponding element in $\mathrm{PSL}(2,\bZ)$.
The eigenvalues of $\cN(F)$ are $\rho(\cN(F))>1$ and $1>\rho(\cN(F))^{-1}>0$, each with a Jordan block of size one.
Hence
\[
h^\cat(F)=\log \rho(\cN(F))>0 \text{ and } h^\pol(F)=s(\cN(F))=0.
\]
Moreover, such autoequivalence is pseudo-Anosov in the sense of \cite{FFHKL}, by \cite[Theorem~3.1]{FFHKL}.
\end{proof}

\begin{rmk}[Discontinuity of categorical polynomial entropy functions]
The spherical twists $T_1$ and $T_2$ on this quiver $3$-Calabi--Yau category provide examples of autoequivalences with discontinuous categorical polynomial entropy functions.
Indeed, by Proposition \ref{prop:SphercialTwist} $(iii)$,
\[
h_t^\pol(T_1)=h_t^\pol(T_2)=0 \text{ for } t<0,
\]
and since $T_1$ and $T_2$ are parabolic, we have
\[
h_0^\pol(T_1)=h_0^\pol(T_2)=1.
\]
Hence the categorical polynomial entropy functions of $T_1$ and $T_2$ are discontinuous at $t=0$.
\end{rmk}

\subsection{P-twists}
\label{sec:P-twists}
As an analogue of spherical objects, Huybrechts--Thomas \cite{HuybrechtsThomas} studied the so-called $\mathbb{P}$-objects. Recall that given a saturated triangulated category $\cD$ endowed with a Serre functor $\mathbf{S}$, an object $\mathcal{E}$ in $\cD$ is called a $\mathbb{P}^d$-\textit{object} for some positive integer $d$, if $\mathbf{S}(\mathcal{E})\cong\mathcal{E}[2d]$ and $\Hom_\cD(\mathcal{E}, \mathcal{E}[*])\cong H^*(\mathbb{P}^d_\mathbb{C}, \mathbb{Z})\otimes \k$ as $\k$-algebras. Examples of $\mathbb{P}$-objects include line bundles in the derived category of a projective hyper-K\"ahler manifold, the structure sheaf of an embedded $\mathbb{P}^d$ inside a $2d$-dimensional holomorphic symplectic variety, etc. 

Similar to spherical twits, any $\mathbb{P}$-object $\mathcal{E}$ also induces an autoequivalence $P_\mathcal{E}$ of the triangulated category \cite{HuybrechtsThomas},  called $\mathbb{P}$-\textit{twists}. We refer to \cite{HuybrechtsThomas} for the precise definition and properties. Note that $P_\mathcal{E}(\mathcal{E})\cong \mathcal{E}[-2d]$ and $P_\mathcal{E}|_{\mathcal{E}^\perp}\cong \operatorname{id}.$

As an analogue of Lemma \ref{lemma:UpperBound}, we have the following estimate, which improves upon \cite[Proof of Theorem  3.1]{FanPtwists}.

\begin{lemma}[Upper bound]\label{lemma:UpperBound2}
For any objects $G, G'\in \cD$ and any positive integer $n$, we have
\[\epsilon_t(G', P^n_\mathcal{E}(G))\leq 
\begin{cases}
nA_t+B_t &\text{ if } t=0;\\
\frac{e^{-2dnt}}{e^{-2dt}-1}A_t+B_t &\text{ if } t< 0;\\
\frac{e^t}{1-e^{-2dt}} A_t+B_t &\text{ if } t>0,\\
\end{cases}\]
where $A_t=\epsilon_t(G', C)$ and $B_t=\epsilon_t(G', G)$ are positively valued functions (independent  of $n$).  Here $C:=\operatorname{Cone}(\Hom(\mathcal{E}, G[*-2])\otimes \mathcal{E}\to\Hom(\mathcal{E}, G[*])\otimes \mathcal{E})$.
\end{lemma}
\begin{proof}
The proof is similar to Lemma \ref{lemma:UpperBound}. Applying $P_\mathcal{E}^{n-1}$ to the distinguished triangle 
\[C\to G\to P_\mathcal{E}(G)\xrightarrow{+1},\]
we get 
\[C[-2d(n-1)]\to P^{n-1}_\mathcal{E}(G)\to P^n_\mathcal{E}(G)\xrightarrow{+1}.\]
Therefore,
$\epsilon_t(G', P^n_\mathcal{E}(G))\leq \epsilon_t(G',  P^{n-1}_\mathcal{E}(G))+A_te^{(1-2d(n-1))t}$,
hence
\[\epsilon_t(G', P^n_\mathcal{E}(G))\leq B_t+A_t\sum_{i=0}^{n-1}e^{(1-2di)t}.\]
The case $t=0$ follows immediately. For $t\neq 0$, we get 
\[\epsilon_t(G', P^n_\mathcal{E}(G))\leq B_t+A_t\frac{e^t}{e^{-2dt}-1}(e^{-2dtn}-1).\]
One can conclude easily by separating  the cases $t>0$ and $t<0$.
\end{proof}

\begin{prop}\label{prop:PTwist}
Let $d$ be a positive integer.  Let $\mathcal{E}$ be a $\mathbb{P}^d$-object in a saturated triangulated category $\cD$. Then
\begin{enumerate}[label=(\roman*)]
    \item If $t<0$, then $\ph_t(P_\mathcal{E})=0$;
    \item If $t>0$ and $\mathcal{E}^\perp \neq 0$, then $\ph_t(P_\mathcal{E})=0$;
    \item If $t=0$, then $0\leq \ph(P_\mathcal{E})\leq 1$.
\end{enumerate}
\end{prop}
\begin{proof}
Once we have Lemma \ref{lemma:UpperBound2}, the proof goes exactly as in Proposition \ref{prop:SphercialTwist}. We leave the details to the reader.
\end{proof}

\subsection{Autoequivalences for curves}
In this section, we study categorical polynomial entropy of autoequivalences of derived categories of smooth projective curves. The discussion splits into two parts: standard autoequivalences (which covers the cases of non-elliptic curves), and elliptic curves (where the Fourier--Mukai transform plays an essential  role).
\subsubsection{Standard autoequivalences}
\begin{prop}\label{prop:StarndardAutCurve}
Let $C$ be a smooth projective curve defined over an algebraically closed field $\k$. Let $F$ be a \textit{standard} autoequivalence of $\cD^b(C)$, namely, it is of the form $F=f^*(-\otimes L)[m]$ for some $f\in \Aut(C)$, $L\in \Pic(C)$, and $m\in \mathbb{Z}$. Then $\ph_t(F)$ is a constant function in $t$ with value
\[\ph(F)=\begin{cases}
0 & \text{ if } \deg(L)=0;\\
1 & \text{ if } \deg(L)\neq 0.
\end{cases}\]
In particular, $\ph(F)$ coincides with the polynomial growth rate of the induced action the Hochschild homology or the numerical Grothendieck group:
\[\ph(F)=s(\mathcal{N}(F)).\]
\end{prop}

\begin{proof}
By Lemma \ref{lemma:Shifts}, we can assume that $m=0$, hence $F=f^*(-\otimes  L)$. Since $F$ preserves the standard t-structure, $\ph_t(F)$ is constant in $t$ by Lemma \ref{lemm:ConstantFunction}. Fix an ample  line bundle $\mathcal{O}(1)$ on $C$ and consider a split generator $G=\mathcal{O}(1)\oplus\mathcal{O}(2)$. As is observed in \cite[Proof of Proposition 3.3]{KiTa}, if $L$ has positive degree, then for $n$ large enough, $G^\vee\otimes F^n(G^\vee)$ is a direct sum of line bundles of positive degree; if $L$ has non-positive degree,  then $G^\vee\otimes F^n(G^\vee)$ is a direct sum of line bundles of negative degree. Therefore, using the Riemann--Roch formula, 
\begin{align*}
\epsilon(G, F^n(G^\vee))&=|\chi(C, G^\vee\otimes F^n(G^\vee)|\\
&= |\deg(G^\vee\otimes F^n(G^\vee))+4(1-g)|\\
&= |\deg(G^\vee\otimes G^\vee\otimes L^{\otimes n})+4(1-g)|,
\end{align*}
which is constant in $n$ if $\deg(L)=0$ and has linear growth in $n$ if $\deg(L)\neq 0$.
\end{proof}

\subsubsection{Elliptic curves}
Let $(E, x_0)$ be a smooth projective curve of genus 1 defined over an algebraically closed field $\k$, together with a closed point $x_0\in E$.
Define autoequivalences
\[
T \coloneqq (- \otimes \cO(x_0)) \text{ \ and \ }
S \coloneqq \Phi_\cP,
\]
where $\Phi_\cP$ is the Fourier--Mukai transform along the Poincar\'e line bundle $\cP\in\Coh(E\times E)$.
The natural map
\[
\phi\colon \Aut(\cD^b(E)) \ra \Aut(\mathcal{N}(E),\chi) \cong \mathrm{SL}(2,\bZ)
\]
is surjective, which sends
\[
T \mapsto
\begin{pmatrix}
1 & 0\\
1 & 1
\end{pmatrix}
\text{ \ and \ }
S \mapsto
\begin{pmatrix}
0 & 1\\
-1 & 0
\end{pmatrix}
\]
with respect to the basis $\{[\cO_E],[\cO_{x_0}]\}$ of the numerical Grothendieck group $\mathcal{N}(E)$. Here $\Aut(\mathcal{N}(E),\chi)$ is the group of isometries of  $\mathcal{N}(E)$ with respect to the Euler pairing $\chi$.
It is well-known that there is a short exact sequence
\[
1 \ra \Aut(E)\ltimes(\Pic^0(E)\times\bZ[2]) \ra \Aut(\cD^b(E)) \xrightarrow{\phi}
\mathrm{SL}(2,\bZ) \ra 1.
\]

\begin{lemma}
\label{lemma:ellcurvethroughSL2}
The map
\[
h^\pol\colon \Aut(\cD^b(E)) \ra [-\infty,\infty]
\]
factors through $\phi\colon\Aut(\cD^b(E))\ra\mathrm{SL}(2, \bZ)$.
\end{lemma}

\begin{proof}
We follow the same idea of the proof of \cite[Lemma~3.4]{Kikuta}.
The aim is to show that if
\[
F = F' g
\]
for some $F,F'\in\Aut(\cD^b(E))$ and $g\in\Aut(E)\ltimes(\Pic^0(E)\times\bZ[2])$,
then
\[
h^\pol(F) = h^\pol(F').
\]
Since $\Aut(E)\ltimes(\Pic^0(E)\times\bZ[2])$ is a normal subgroup in $\Aut(\cD^b(E))$, for each $n$ there exists $g_n\in\Aut(E)\ltimes(\Pic^0(E)\times\bZ[2])$ such that
\[
F^n = F'^n g_n.
\]
Fix an ample line bundle $\cO(1)$ on $E$ and let $G=\cO(1)\oplus\cO(2)$ be a split generator.
Then
\begin{align*}
    \delta(G,F^nG^\vee) & \leq \delta(G, F'^nG)\delta(F'^nG,F^nG^\vee) \\
    &  = \delta(G, F'^nG)\delta(F'^nG,F'^ng_nG^\vee) \\
    & \leq \delta(G, F'^nG)\delta(G, g_nG^\vee)
\end{align*}
By \cite[Lemma~3.4]{Kikuta}, we have $h^\cat(F)=h^\cat(F')$. Therefore,
\begin{equation}
\label{eq:ellcurveSL2}
    \frac{\log\delta(G,F^nG^\vee)-nh^\cat(F)}{\log n} 
    \leq \frac{\log\delta(G, F'^nG)-nh^\cat(F')}{\log n} +
    \frac{\log\delta(G, g_nG^\vee)}{\log n}.
\end{equation}
The argument in the proof of \cite[Lemma~3.4]{Kikuta} shows that
\[
\lim_{n\ra\infty}\frac{\log\delta(G, g_nG^\vee)}{\log n} = 
\lim_{n\ra\infty}\frac{\log\ep(G, g_nG^\vee)}{\log n} = 
\lim_{n\ra\infty}\frac{\log|\chi(G, G^\vee)|}{\log n} =0.
\]
Taking limit $n\ra\infty$ of \eqref{eq:ellcurveSL2} gives $h^\pol(F)\leq h^\pol(F')$.
One can prove $h^\pol(F')\leq h^\pol(F)$ using the same argument.
Hence we have $h^\pol(F)=h^\pol(F')$.
\end{proof}

The main  result of this section is the following trichotomy, where the statements concerning categorical entropy is due to Kikuta \cite[Section 3.2]{Kikuta}. One sees clearly how  polynomial entropy further refines his study.
\begin{thm}
\label{thm:ellcurveclassify}
Let $F\in \Aut(\cD^b(E))$. We have
    \begin{enumerate}[label=(\roman*)]
        \item $h^\pol(F)=h^\cat(F)=0$ if and only if
        $\cN(F)$ is elliptic (i.e.~$|\tr(\cN(F))|<2$) or 
        $\cN(F)=\pm\mathrm{id}$.
        
        \item $h^\pol(F)>0$ and $h^\cat(F)=0$ if and only if
        $\cN(F)$ is parabolic (i.e.~$|\tr(\cN(F))|=2$)
        and $\cN(F)\neq\pm\mathrm{id}$. In this case, $\ph(F)=1$.
        
        \item $h^\cat(F)>0$ if and only if 
        $\cN(F)$ is hyperbolic (i.e.~$|\tr(\cN(F))|>2$).
        In this case, $h^\pol(F)=0$.
    \end{enumerate}
Moreover, we have the following Gromov--Yomdin-type equality for the categorical polynomial entropy:
\[
h^\pol(F) = s (\cN(F)).
\]
\end{thm}

\begin{proof}
$(i)$. An elliptic element in $\mathrm{SL}(2,\bZ)$ is conjugate to either
\[
\pm\begin{pmatrix}
0 & 1 \\
-1 & 0
\end{pmatrix}
\text{, }
\pm\begin{pmatrix}
1 & 1 \\
-1 & 0
\end{pmatrix}
\text{, or }
\pm\begin{pmatrix}
0 & 1 \\
-1 & -1
\end{pmatrix}.
\]
Observe that
\[
\phi(S)=
\begin{pmatrix}
0 & 1 \\
-1 & 0
\end{pmatrix},
\ 
\phi(ST)=
\begin{pmatrix}
1 & 1 \\
-1 & 0
\end{pmatrix},
\text{ and }
\phi((ST)^2)=
\begin{pmatrix}
0 & 1 \\
-1 & -1
\end{pmatrix}.
\]
Hence if $\cN(F)$ is elliptic, by Lemma~\ref{lemma:ellcurvethroughSL2}, we have
\[
h^\pol(F)=h^\pol(F'),
\]
where $F'\in\{S, ST, (ST)^2\}$.
Since $S, ST, (ST)^2$ are all of finite order up to shifts \cite[Section 3d]{SeidelThomas}, their categorical and polynomial entropy both vanish.
Hence
\[
h^\pol(F)=h^\cat(F)=0
\]
for autoequivalences $F$ such that $\cN(F)$ is elliptic.
Moreover, in this case $\cN(F)$ is of finite order, hence
\[
h^\pol(F)=s(\cN(F))=0.
\]
$(ii)$. A parabolic element in $\mathrm{SL}(2,\bZ)$ is conjugate to
\[
\pm\begin{pmatrix}
1 & 0 \\
n & 1
\end{pmatrix}
\]
for some $n\in\bZ$.
Hence if $\cN(F)$ is parabolic, by Lemma~\ref{lemma:ellcurvethroughSL2},
\[
\phi(F) = \phi(T^n)
\]
for some $n\in\bZ$.
By Proposition \ref{prop:StarndardAutCurve}, we have
\[
h^\cat(F)=0 \text{ \ and \ } h^\pol(F)=1
\]
if $n\neq0$.
Moreover, in this case $\cN(F)$ is quasi-unipotent and has a single Jordan block of size $2$. Therefore, we have
\[
h^\pol(F)=s(\cN(F))=1.
\]
$(iii)$. Suppose $\cN(F)$ is hyperbolic.
By \cite[Proposition~3.9]{Kikuta}, we have
\[
h^\cat(F) = \log \rho(\mathcal{N}(F)) > 0.
\]
Moreover, using the argument of the proof of \cite[Proposition~3.9]{Kikuta}, we have
\begin{align*}
h^\pol(F) & = \limsup_{n\ra\infty} \frac{\log\ep(G^\vee,F^nG)-nh^\cat(F)}{\log n} \\
& = \limsup_{n\ra\infty}\frac{\log|\chi(G\otimes F^nG)|-n\log\rho(\mathcal{N}(F))}{\log n}.
\end{align*}
Since
\[
\chi(G\otimes F^nG) = A \rho(\mathcal{N}(F))^n + B\rho(\mathcal{N}(F))^{-n}
\]
for some constant $A,B$ with $A\neq0$, we obtain that $\ph(F)=0$.
Moreover, in this case $\cN(F)$ has eigenvalues $\rho(\mathcal{N}(F))>1$ and $1>\rho(\mathcal{N}(F))^{-1}>0$. Each of the eigenvalues has a single Jordan block of size one. Hence
$h^\pol(F)=s(\cN(F))=0.$
\end{proof}




\bigskip
\bibliography{cat_poly_entropy}
\bibliographystyle{abbrv}

\medskip
\emph{Yu-Wei Fan}, \texttt{ywfan@berkeley.edu} \\
\textsc{University of California at Berkeley, 735 Evans Hall, Berkeley, CA 94270, USA}\\

\emph{Lie Fu}, \texttt{fu@math.univ-lyon1.fr} \\
\textsc{Universit\'e Claude Bernard Lyon 1, 43 bd du 11 novembre 1918, 69622 Ville\-urbanne cedex, France}\\
$\&$\\
\textsc{Radboud University, Heyendaalseweg 135, 6525 AJ Nij\-megen, Netherlands}\\

\emph{Genki Ouchi},  \texttt{genki.ouchi@riken.jp},\\ \textsc{Interdisciplinary Theoretical and Mathematical Sciences Program, RIKEN, 2-1 Hirosawa, Wako, Saitama, 351-0198, Japan}  \\
\textsc{}

\end{document}